\documentclass[a4paper,10pt]{article}
\usepackage{a4wide}
\usepackage{amsmath,amssymb}
\usepackage[english]{babel}
\usepackage{subfig}
\usepackage{multirow}
\usepackage{verbatim}
\usepackage{float}
\usepackage[section]{placeins}
\usepackage{graphicx}
\usepackage{pdfpages}
\usepackage{xargs}
\usepackage{mathtools}
\usepackage{dsfont}                     % Use more than one optional parameter in a new commands
\usepackage{textcomp}
\usepackage[colorinlistoftodos,prependcaption,textsize=tiny]{todonotes}
\usepackage{pgfplots}
\pgfplotsset{compat = 1.14}
\usepackage{tikz-network}
\usetikzlibrary{positioning}
\usepackage[title]{appendix}

\usepackage{booktabs}
\usepackage{cite}
\usepackage[square,sort,comma,numbers]{natbib}

\usepackage{pifont}
\DeclarePairedDelimiter{\ceil}{\lceil}{\rceil}
\DeclarePairedDelimiter{\floor}{\lfloor}{\rfloor}

\usepackage{parskip}
\definecolor{TomColBack}{HTML}{F62217}
\definecolor{TomCol}{HTML}{16EAF5}

\usepackage{fancyhdr} %Footers en Headers
\usepackage{tikz}
\usetikzlibrary{arrows}
\usetikzlibrary{positioning}
\usetikzlibrary{decorations}

\definecolor{mylilas}{HTML}{CC0099}
\usepackage{listings}
\definecolor{mygreen}{rgb}{0,0.6,0} %groene kleur voor comments in code
\definecolor{mybackground}{HTML}{F7BE81}
\usepackage{caption}
\DeclareCaptionFont{white}{\color{white}}
\DeclareCaptionFormat{listing}{\colorbox{gray}{\parbox{\textwidth}{#1#2#3}}}
\usepackage{hyperref}
\hypersetup{
    pagebackref=true,
    colorlinks=true,
    citecolor=red,
    linkcolor=blue,
    linktoc=page,
    urlcolor=blue,
    pdfauthor= Mark Christianen,
    pdftitle=Stability-EV
}
\usepackage{aligned-overset}
\definecolor{redText}{rgb}{1,0,0}

\definecolor{blueText}{HTML}{0080FF}

\definecolor{greenText}{HTML}{00e600}

\usetikzlibrary{chains,shapes.multipart}
\usetikzlibrary{shapes,calc}
\usetikzlibrary{automata,positioning}
\usepackage{cases}

\definecolor{myred}{RGB}{220,43,25}
\definecolor{mygreen}{RGB}{0,146,64}
\definecolor{myblue}{RGB}{0,143,224}

\definecolor{MarkCol}{HTML}{B216FA}
\definecolor{MarkColBack}{HTML}{5FFB17}
\newcommandx{\Mark}[2][1=]{\todo[linecolor = MarkCol, backgroundcolor = MarkColBack!30, bordercolor = MarkColBack,#1]{#2}}

\tikzset{
myshape/.style={
  rectangle split,
  minimum height=1.5cm,
  rectangle split horizontal,
  rectangle split parts=8,
  draw,
  anchor=center,
  },
mytri/.style={
  draw,
  shape=isosceles triangle,
  isosceles triangle apex angle=60,
  inner xsep=0.65cm
  }
}

\usepackage{algorithm}
\usepackage[noend]{algpseudocode}
\usepackage{amsthm}
\newtheorem{theorem}{Theorem}[section]
  \usepackage{titling}
\predate{}
\postdate{}
\newtheorem{lemma}{Lemma}[section]
\newtheorem{proposition}{Proposition}[section]
\newtheorem{corollary}{Corollary}[section]

\newtheorem{remark}{Remark}[section]
\newtheorem{assumption}{Assumption}[section]
\DeclareMathOperator*{\argmax}{arg\,max}

\usepackage{color, colortbl}
\usepackage{authblk}
\numberwithin{equation}{section}
\usepackage{enumitem}
\newtheorem{definition}{Definition}

\newcommand{\Sup}[1]{\raisebox{0.5ex}{\scalebox{0.8}{$\displaystyle \sup_{#1}\;$}}}

\begin{document}

\author[1]{\small Christianen, M.H.M.}
\author[4]{\small Cruise, J.}
\author[1]{\small Janssen, A.J.E.M.}
\author[4]{\small Shneer, S.}
\author[1,2]{\small Vlasiou, M.}
\author[1,3]{\small Zwart, B.}

\affil[1]{\footnotesize Eindhoven University of Technology}
\affil[2]{\footnotesize University of Twente}
\affil[3]{\footnotesize Centrum Wiskunde \& Informatica}
\affil[4]{\footnotesize Heriot-Watt University}

\title{Comparison of stability regions for a line distribution network with stochastic load demands}
\date{}
\maketitle

\begin{abstract}
%% Text of abstract
We compare stability regions for different power flow models in the process of charging electric vehicles (EVs) by considering their random arrivals, their stochastic demand for energy at charging stations, and the characteristics of the electricity distribution network. We assume the distribution network is a line with charging stations located on it. We consider the \emph{Distflow} and the \emph{Linearized Distflow} models %The power is allocated to the EVs such that it meets its stochastic demand for energy and voltage drops stay bounded and the voltages meet the recursive equations from the Distflow and Linearized Distflow model.
and we assume that EVs have an exponential charging requirement, that voltage drops on the distribution network stay under control and that the number of charging stations $N$ goes to infinity. We investigate the stability of utility-optimizing power allocations in large distribution networks for both power flow models by controlling the arrival rate of EVs to charging stations. For both power flow models, we show that, to obtain stability, the maximum feasible arrival rate, i.e. stability region of vehicles is decaying as $1/N^2$, and the difference between those arrival rates is up to constants, which we compare explicitly.
\end{abstract}

\section{Introduction}
%\emph{Current problem with EV-charging.}
The popularity of electric vehicles (EVs) has been growing due to their increased range, lower cost of batteries, and governmental subsidies. However, EVs need to be charged and the current infrastructure cannot support their increasing demand for energy. This can therefore cause capacity problems in distribution networks in the (very near) future \cite{Hoogsteen2017}. Modifying existing infrastructure is costly and constrained by the limits of the distribution network. Thus, we need a mechanism that guarantees the quality of service provided to EV drivers, subject to distribution network constraints.

%\emph{Real life EV-charging.}
Motivated by this, we consider EVs charging in a neighborhood such that voltage drops on the distribution network stay under control. EVs arrive randomly (at charging stations) to get charged and have different stochastic demands for energy. Moreover, the charging rates allocated to EVs depend on the number of cars charging simultaneously in the neighborhood. We model this process as a queue, with EVs representing \emph{jobs}, and charging stations classified as \emph{servers}, constrained by the physical limitations of the distribution network.

We are interested in conditions guaranteeing stability of the queuing model under distribution network constraints. Whenever we write stability, we mean stability of the queuing model, unless stated otherwise. In our setting, we assume that cars arrive at the same rate at each charging station and we want to find the maximum feasible arrival rate in terms of the number of charging stations $N$ such that voltage drops are within constraints. This is very challenging due to the uncertainty of the stochastic electricity demand of cars. Another difficult aspect is the modeling of the power flow dynamics in the distribution network. In this paper, we approximate the power flow dynamics in the distribution network by the \emph{Distflow} and the \emph{Linearized Distflow} models. The main stability results (Theorems \ref{THM:LINDIST} and \ref{THM:DISTFLOW}) show, on the one hand, that for large distribution networks (to obtain stability) the maximal feasible arrival rate is decaying as $1/N^2$, and on the other hand that the difference between these arrival rates is up to constants, which we compare explicitly.
%In the rest of the introduction we discuss distribution networks, power flow models and existing stability conditions for EV charging.
%\emph{Distribution network in real life}
\paragraph{Distribution networks}
An electric grid is a connected network that transfers electricity from producers to consumers. It consists of generating stations that produce electric power, high voltage transmission lines that carry power from distant sources to demand centers, and distribution lines that connect individual customers, e.g., houses, charging stations, etc. We focus on a network that connects a generator to charging stations with only distribution lines. Such a network is called a \emph{distribution network}.

%\emph{Voltage drop constraint}
An important constraint in a distribution network is the requirement of keeping voltage drops on a line under control. Voltage drop is the difference between voltages at subsequent charging stations. Distribution lines have an impedance, which results to voltage loss during transportation. The maximum allowed voltage drop ensures that every customer receives safe and reliable energy at a voltage that is within some standard range, which varies from one country to another\ \cite{Kerstinga}.

%\emph{Model distribution network}
In the rest of the paper, we assume that the distribution network, consisting of one generator, several charging stations and distribution lines, has a line topology. The generator that produces electricity is called the \emph{root node}. Charging stations consume power and are called the \emph{load nodes}. Thus, we represent the distribution network by a graph (here, a line) with a root node, load nodes, and edges representing the distribution lines. Actually, the model does not depend on the load demand being from EV drivers to charge their cars or not. Any stochastic load demand fits this framework.

%\emph{Power flow model}
\paragraph{Power flow models} In order to model the power flow in the network, we use approximations of the alternating current (AC) power flow equations \cite{Molzahn2019}. These power flow equations characterize the steady-state relationship between power injections at each node, the voltage magnitudes, and phase angles that are necessary to transmit power from generators to load nodes. First, we study a load flow model known as the \emph{branch flow model} or the \emph{Distflow model} \cite{Low2014d,Baran1989} and additionally, a linearized version of the Distflow model called the \emph{Linearized Distflow} model. Both power flow models focus on quantities such as complex current and power flowing on the distribution lines. Due to the specific choice for the distribution network as a line, both power flow models have a recursive structure, which we can exploit. The accuracy and effectiveness of the Linearized Distflow model has been numerically justified in the literature \cite{Baran1989}. Its use is justified by the fact that the nonlinear terms in the equations of the Distflow model represent the losses that appear if electric power is transferred over the lines. These losses should be, in practice, much smaller than the active and reactive power terms that show up in the equations. %Furthermore, we assume that EVs consume active power only \cite{Carvalho2015b}.
In this paper, we show that the impact of neglecting these losses, in terms of stability, are insignificant if the network is large.

%\emph{Real life queue}
\paragraph{Stability conditions for EV charging}
Customers can charge their EVs at charging stations. In practice, EVs are served simultaneously, because they require concurrent usage of all upstream distribution lines between the location of an EV and the generator of the distribution network. We can adequately model the arrival and charging of cars as a resource-sharing network by the use of \emph{bandwidth-sharing networks}. Bandwidth-sharing networks are a specific class of queuing networks where the service capacity is shared among all concurrent users. Bandwidth-sharing networks have been successfully applied in communication networks \cite{Massoulie1999}.

%\emph{Literature on EV-charging and stability}
The literature on stability results for EV-charging is limited to numerical experiments. An early paper on stability analysis in EV-charging is \cite{Huang2013}, which presents a new quasi-Monte Carlo stability analysis method to assess the dynamic effects of plug-in electric vehicles in power systems. Huang et al. conclude that improvements for stability control are worth further study since the number of EVs is growing.

Simulation  studies are conducted to obtain stability conditions in \cite{Carvalho2015b}. Here, Carvalho et al. find that there is a threshold value $\lambda_c$ on the arrival rates, such that if the actual arrival rate $\lambda$ is greater than this threshold, i.e. $\lambda>\lambda_c$, some vehicles have to wait for increasingly long times to fully charge. Similar findings were obtained in \cite{Zhang2016,DeHoog2014,Ul-Haq2015,Dharmakeerthi2014}. Up to a certain point, the distribution network is able to serve all EVs properly, but if the number of EVs in the system is too high, it cannot be guaranteed that, e.g., the voltage drops stay under control.

The present paper builds upon \cite{Aveklouris2019}. In this study, Aveklouris et al. consider a distribution network used to charge EVs such that voltage drops stay under control, taking into account randomness of future arriving EVs and power demands. The work focuses on a fluid approximation for the number of uncharged EVs, while we focus on conditions to ensure stability of the queuing model as in \cite{Shneer2018,Shneer2019c}. These studies establish stability for networks of interacting queues governed by utility-maximizing service-rate allocations using direct applications of Lyapunov-Foster-type criteria. Our queuing network fits their general model framework, so using results from \cite{Shneer2018,Shneer2019c}, provides a way to obtain stability conditions for our queuing model. Furthermore, we use differential and integral calculus to study the difference between maximal feasible arrival rates obtained under the Linearized Distflow and the Distflow models. The main conclusion that we draw is that these rates are the same as the size of the distribution network increases, up to constants.

The structure of the paper is as follows. In Section \ref{sec:model_description}, we provide a detailed model description. In particular, we introduce the queuing model, the distribution network model and the power flow models. In Section \ref{sec:main_results}, we find stability regions for both power flow models. The stability results are presented in Sections \ref{subsec:lindist} and \ref{subsec:distflow}. The comparison between the stability regions of both power flow models is made in Section \ref{SUBSEC:COMPARISON}. The aforementioned stability results under the Linearized Distflow and the Distflow models are proven in Sections \ref{section:proof_lindist} and \ref{section:proof_distflow}, respectively.  From an engineering point of view, the computation of the stability region for the Distflow model, as in Section \ref{subsec:distflow}, can also be done numerically, via an iterative approach. Therefore, we show how to use Newton's method to compute the maximal feasible arrival rate under the Distflow model in Section \ref{sec:improvement}. The rest of the paper focuses on proofs and a theorem that are used in Sections \ref{subsec:distflow}, \ref{section:proof_lindist} and \ref{section:proof_distflow}.

In Section \ref{subsec:power_flow}, we derive ways to establish if a given power allocation is feasible under both power flow models. For the Linearized Distflow model we can directly use a recursion to compute the voltages at all nodes, however, for the Distflow model this is not possible. Therefore, in Appendix \ref{sec:equivalence}, we prove an equivalence for the voltages (under the Distflow model) at the root node and the node at the end of the line, as in \cite{Vasmel2019}, such that we can use a recursion to compute the voltages at all nodes. Furthermore, our paper uses results from \cite{Shneer2018,Shneer2019c}. Therefore, in Appendix \ref{sec:stability_result}, we consider the general model framework in \cite{Shneer2018} and show that we can drop one of their assumptions and still apply their theorem indicating stability, which we used in Sections \ref{subsec:lindist} and \ref{subsec:distflow}. Further, an important step in the proof in Section \ref{section:proof_distflow} is the approximation of a scaled version of the voltages under the Distflow model in Section \ref{subsec:approximation_continuous counterpart}. In
Appendix \ref{sec:conv_VN_to_V(1)}, we present the proof of the convergence of a scaled version of the voltages under the Distflow model to the solution of an integral equation and the numerical validation of this convergence. Furthermore, the solution to this integral equation is given in Appendix \ref{sec:integral_equation}. Last, in Appendix \ref{sec:appendix_comparison}, we proof the result to compare the stability regions under both power flow models.

In the remainder of this section, we list the notation we use throughout the paper. An overview of notations is also given in Appendix \ref{sec:notation}.

\paragraph{Notation} %To define $x$ as equal to $y$, we write $x:=y$ or $y:=x$. We abbreviate the left-hand side and right-hand side of an equation as ``left-hand side'' and ``right-hand side'', respectively.

%The standard sets are: the natural numbers $\mathbb{N}:=\{1,2,\ldots\}$, the integers $\mathbb{Z}:=\{0,\pm 1,\pm 2,\ldots\}$, the non-negative integers $\mathbb{Z}_{+} = \{0\}\cup \mathbb{N}$, the real line $\mathbb{R}:=(-\infty,\infty)$ and the non-negative half-line $\mathbb{R}_{+}:=[0,\infty)$.

All vectors and matrices are denoted by bold letters. Vector inequalities hold coordinate-wise; namely, $\mathbf{x}>\mathbf{y}$ implies that $x_i>y_i$ for all $i$. The imaginary unit is denoted by $\mathrm{i}$ and the absolute value of a complex number $z=x+\mathrm{i} y$ is $\left|z\right|=\sqrt{x^2+y^2}$. Also, the complex conjugate $x-\mathrm{i}y$ of $z$ is denoted by $\overline{z}$.

The following operations are defined on $x,y\in\mathbb{R}$:
\begin{align*}
\floor{x} &:=\max\{m\in\mathbb{Z}:m\leq x\}, \\
\ceil{x} &:=\min\{n\in\mathbb{Z}:n\geq x\}.
\end{align*}
%Let $\mathbb{R}$ be endowed with metric $d:\mathbb{R}\times \mathbb{R}\to \mathbb{R}_+$ defined by
%\begin{align}
%d(f,g)=\sup_{t\in[0,T]}|f(t)-g(t)|,
%\end{align} and
Denote the space of functions $f:[0,T]\to\mathbb{R}$ that are right-continuous with left limits, i.e. c\`{a}dl\`{a}g functions, by
\begin{align*}
\mathbf{D}[0,T].
\end{align*} Furthermore, we define the space $\mathbf{D}_{\geq 1}[0,T]$ as %the space of c\`{a}dl\`{a}g functions such that $\inf_{t\in[0,T]}f(t)\geq 1$; i.e.,
\begin{align*}
\mathbf{D}_{\geq 1}[0,T] := \{f\in \mathbf{D}[0,T]: \inf_{t\in[0,T]} f(t)\geq 1.\}.
\end{align*}
For a function $f(\cdot)$ defined on (a subset of) $\mathbb{R}$, $f'(t)$ denotes its derivative at $t$ and $f''(t)$ its second derivative at $t$, when these exist.

\section{Model description}\label{sec:model_description}
This section describes the three main components of the EV-charging model. In Section \ref{subsec:queueing}, we describe the characteristics of the queuing model; i.e., the evolution of the number of cars charging at each charging station. In Section \ref{subsec:distribution}, we specify the distribution network model and in Section \ref{subsec:power_flow}, we introduce the Distflow and the Linearized Distflow models.

\subsection{Queuing model of EV-charging}\label{subsec:queueing}
We use a queuing model to study the process of charging EVs in a distribution network. In this setting, EVs, referred to as jobs, require service. This service is delivered by charging stations, referred to as servers. The service being delivered is the power supplied to EVs.

In the queuing system, we consider $N$ single-server queues, each having its own arrival stream of jobs. Denote by $\mathbf{X}(t) = (X_1(t),\ldots,X_N(t))$ the vector giving the number of jobs at each queue at time $t$. We make the following assumption on the arrival rates and service requirements of all EVs.
\begin{assumption}\label{assumption:arrival_rate}
At all charging stations, all EVs arrive independently according to Poisson processes with the same rate $\lambda$ and have independent service requirements which are $Exp(1)$ random variables.
\end{assumption} %Assumption \ref{assumption:arrival_rate} %is probably not the most pragmatic and realistic assumption, however we make this assumption to
%keeps the analysis in Section \ref{subsec:distflow} tractable.

\begin{remark}\label{remark:equal_arrival_rates} In Appendix \ref{sec:stability_result}, we state the stability condition as in \cite{Shneer2018}. The proof is given for general arrival rates, however in order to compare the stability regions under both power flow models, we focus on equal arrival rates for all nodes.
\end{remark}
At each queue, all jobs are served simultaneously and start service immediately (there is no queuing). Furthermore, each job receives an equal fraction of the service capacity allocated to a queue. Denote by $\mathbf{\tilde{p}}(t) = (\tilde{p}_1(t),\ldots,\tilde{p}_N(t))$ the vector of service capacities allocated to each queue at time $t$. This represents the active power that is allocated to each node. This means that at each queue $j$, it takes 1/$\tilde{p}_j(t)$ time units to serve one job. In our model however, service capacities are state-dependent and subject to changes, and the dynamics are more complicated. See below for details.

We can then represent the number of electric vehicles charging at every station as an $N$-dimensional continuous-time Markov process. The evolution of the queue at node $j$ is given by
\begin{align*}
X_j(t) \to X_j(t)+1~\text{at rate}~\lambda
\end{align*} and
\begin{align*}
X_j(t) \to X_j(t)-1~\text{at rate}~\tilde{p}_j(t).
\end{align*}
From now on, for simplicity, we drop the dependence on time $t$ from the notation. For example, we write $X_j$ and $\tilde{p}_j$ instead of $X_j(t)$ and $\tilde{p}_j(t)$. We assume that the rates $\mathbf{\tilde{p}}$ may be allocated according to the current vector $\mathbf{X}=(X_1,\ldots,X_N)$ of number of jobs.

A popular class of policies in the context of bandwidth-sharing networks are $\alpha$-fair algorithms \cite{Bonald2006a}. In state $\mathbf{X}$, an $\alpha$-fair algorithm allocates $\tilde{p}_j/X_j$ to each EV at charging station $j$, with $\mathbf{\tilde{p}}=(\tilde{p}_1,\ldots,\tilde{p}_N)$ the solution of the utility optimization problem
\begin{align}
\mathbf{\tilde{p}} \in \argmax \sum_{j=1}^N X_j U_j^{(\alpha)}\left(\frac{\tilde{p}_j}{X_j}\right),\label{eq:argmax_p}
\end{align} subject to physical constraints on the vector $\mathbf{\tilde{p}}$ of allocated power and where
\begin{align*}
U_j^{(\alpha)}(x_i) =
\begin{cases}
\log x_i &\ \text{if}\ \alpha=1, \\
x_i^{1-\alpha}/(1-\alpha) & \ \text{if}\ \alpha\in(0,\infty)\backslash\{1\},
\end{cases},\ x_i\geq 0.
\end{align*} The parameter $\alpha$ measures the degree of fairness of the allocation. Popular choices are $\alpha\to 0$, where the total allocated power tends to be maximized, but the allocation is very unfair. The choice $\alpha=1$,  where we end up with proportional fairness, that tends to maximize the utility of the total power allocation or $\alpha=2$, that corresponds to the minimum potential delay allocation, where the total charging time tends to be minimized. Last, we have the limiting case $\alpha\to\infty$, where the minimum of the power allocated to any charging station tends to maximized, namely max-min fairness \cite{Massoulie1999}.

%The rates must be allocated between $N$ different charging stations by some \emph{decision maker}. The decision maker wants to share some function $h$ of the available power $p_j$ at each charging station weighted to some function $g$ of the number of cars $X_j$ at each charging station in the system to maximize overall utility; i.e.,
% where the functions $h$ and $g$ satisfy the following conditions \cite{Shneer2018}. The function $h:[0,\infty)\to\mathbb{R}$ is strictly increasing, differentiable and concave (both the cases $\lim_{y\downarrow 0}h(y)=h(0)>-\infty$ and $\lim_{y\downarrow 0}h(y) = -\infty$ are allowed). The function $g:\mathbb{Z}_+ \to [0,\infty)$ is strictly increasing and such that
%\begin{align}
%\frac{g(y+1)}{g(y)}\to 1\ \text{as}\ y\to\infty.
%\end{align} This last assumption can be interpreted as follows: a small increase in the number of cars at a charging station changes the weight related to the corresponding allocated power to this charging station only slightly. Note that the wide class of functions satisfying the conditions above also contains the so called \emph{$\alpha$-fair algorithms}, $g(y)=y^{\alpha}$ and $h(y)=\frac{y^{1-\alpha}}{1-\alpha}$ with $\alpha>0,\alpha\neq 1$, or $g(y)=y$ and $h(y)=\log(y)$.
\subsection{Distribution network model}\label{subsec:distribution}
The distribution network is modeled as a directed graph $\mathcal{G}=(\mathcal{I},\mathcal{E})$, where we denote by $\mathcal{I} = \{0,\ldots,N\}$ the set of nodes and by $\mathcal{E}$ its set of directed edges, assuming that node $0$ is the root node. We assume that $\mathcal{G}$ has a line topology. Each edge $\epsilon_{j-1,j}\in\mathcal{E}$ represents a line connecting nodes $j-1$ and $j$ where node $j$ is further away from the root node than node $j-1$. Each edge $\epsilon_{j-1,j}\in\mathcal{E}$ is characterized by the impedance $z=r+\mathrm{i} x$, where $r,x\geq 0$ denote the resistance and reactance along the lines, respectively.
Here, we have assumed that the values of the resistance and reactance
along all edges are the same.
\begin{assumption}
All edges have the same resistance value $r$.% and reactance value $x$.
\end{assumption} %Alternatively, the corresponding mutual admittance is $y = g+\mathrm{i}b=1/z$.
We assume that the phase angle between voltages $
\tilde{V}_i$ and $\tilde{V}_j$ is small in distribution networks \cite{Carvalho2015b}, and hence the phases of $
\tilde{V}_i$ and $
\tilde{V}_j$ are approximately the same and can be chosen so that the phasors have zero imaginary components. % In general, there is also the notion of a complex form for the voltage.  However, we assume that voltages have zero imaginary components \cite{Carvalho2015b}. For an explanation why this is a reasonable assumption, we refer to \cite{Kerstinga}.
For $j\in\mathcal{I}$, $\tilde{V}_j$ denotes the real voltage and it emerges that the impedance is zero, thus $z=r$ and all edges have the same resistance value $r$.
\begin{assumption}
The voltage $\tilde{V}_j$ at every charging station $j=1,\ldots,N$ is a positive real number.
\end{assumption}
Furthermore, let $\tilde{s}_j = \tilde{p}_j + \mathrm{i} \tilde{q}_j$ be the complex power consumption at node $j$. Here, $\tilde{p}_j$ and $\tilde{q}_j$ denote the active and reactive power consumption at node $j$, respectively (cf. \eqref{eq:argmax_p}). By convention, a positive active (reactive) power term corresponds to consuming active (reactive) power. Since EVs can only consume active power \cite{Carvalho2015b}, it is natural to make the following assumption.
\begin{assumption}
The active power is non-negative and the reactive power is zero at all charging stations.
\end{assumption}
%\begin{assumption}
%The reactive power consumption is equal to zero at every charging station, i.e.,
%\begin{align*}
%\tilde{q}_j = 0\ \text{for all}\ j=1,\ldots,N.
%\end{align*}
%\end{assumption}
For each $\epsilon_{j-1,j}\in \mathcal{E}$, let $I_{j-1,j}$ be the complex current and $\tilde{S}_{j-1,j}=\tilde{P}_{j-1,j}+\mathrm{i}\tilde{Q}_{j-1,j}$ be the complex power flowing from node $j-1$ to $j$. Here, $\tilde{P}_{j-1,j}$ and $\tilde{Q}_{j-1,j}$ denote the active and reactive power flowing from node $j-1$ to $j$.
The model is illustrated in Figure \ref{fig:model}.
\begin{figure}[h!]
\begin{center}
\begin{tikzpicture}[main_node/.style={circle,fill=blue!10,minimum size=1em,inner sep=4pt},
feeder_node/.style={ellipse,fill=yellow!10,minimum size=3em,inner sep=4pt},]

    \node[feeder_node][label={$\tilde{V}_0$}] (1) at (0,0) {Root node};
    \node[main_node][label={$\tilde{V}_1$}] (2) at (4,0)  {1};
    \node[main_node][label={$\tilde{V}_i$}] (3) at (8,0) {i};
    \node[main_node][label={$\tilde{V}_N$}] (4) at (12,0) {N};
    \node[below=1.2cm of 2](D){};
    \node[below=1.2cm of 3](E){};
    \node[below=1.2cm of 4](F){};

    \draw (1.4,0) -- (3.7,0);
    \draw[dashed] (4.3,0) -- (7.7,0)node[midway,above]{$(\tilde{S}_{ij},I_{ij},r)$};
    \draw[dashed] (8.3,0) -- (11.7,0);
    \Edge[Direct,label=$\lambda$](D)(2);
    \Edge[Direct,label=$\lambda$](E)(3);
    \Edge[Direct,label=$\lambda$](F)(4);
    	%(1) .. node[below] {$R$} (2)
    	%(2) edge node[below] {$R$} (3)
    	%(3) edge node[below] {$R$} (4);
\end{tikzpicture}
\end{center}
\caption{Line network with $N$ charging stations and arriving vehicles at rate $\lambda$.}
\label{fig:model}
\end{figure}
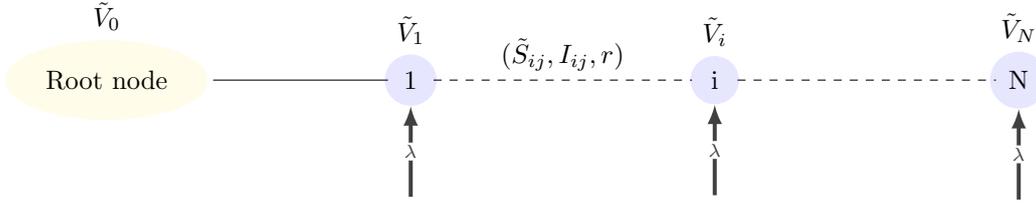

The distribution network constraints are described by a set $\mathcal{C}$. The set $\mathcal{C}$ is contained in an $N$-dimensional vector space and represents feasible power allocations. In our setting, a power allocation is feasible if the maximal voltage drop; i.e., the relative difference between the base voltage $\tilde{V}_0$ and the minimal voltage in all buses between the root node and any other node is bounded by some value $\Delta\in (0,\frac{1}{2}]$; i.e.,
\begin{align}
\frac{\tilde{V}_0-\min_{1\leq j\leq N} \tilde{V}_j}{\tilde{V}_0}\leq \Delta.\label{eq:voltage_drop_C}
\end{align} In Section \ref{subsec:distflow}, we show there is a technical reason why we use the explicit interval $(0,\frac{1}{2}]$ for $\Delta$. However, this is not out of the physical realm, since in practice, we want the maximal voltage drop to be no more than a small percentage of the base voltage, e.g. $\Delta = 0.05$ or $\Delta=0.1$. After the introduction of the power flow models, we give a more concrete definition of the constraint set $\mathcal{C}$ in Section \ref{subsubsec:power_flow_summary}.
   %Furthermore, we make an additional assumption on the arrival rate of EVs, which is necessary for stability. \red{The power quality in the line can be characterized by the maximal voltage drop level - the difference between the base voltage $V_0$ and the minimal voltage in all buses: $\Delta = V_0 - \text{min}_k V_k$}

%\begin{assumption}\label{assumption:arrival_consumption}
%The arrival rate of EVs is smaller than the charging rate at every charging station, i.e.,
%\begin{align}
%\text{there exists}~\mathbf{\tilde{p}}\in \mathcal{C}~ \text{such that}~\boldsymbol{\lambda}<\mathbf{\tilde{p}}.\label{eq:condition}
%\end{align}
%\end{assumption}

\subsection{Power flow models}\label{subsec:power_flow} We introduce two commonly used models to represent the power flow that are valid for radial systems; i.e., systems where all charging stations have only one (and the same) source of supply. % The Distflow equations fully represent the power flows for a balanced, single-phase equivalent model of a radial network. %One obtains a set of equations that fully represent the power flows in mesh networks by augmenting the DistFlow equations with constraints which ensure that there exists phase angles which, 1) are consistent with the squared voltage magnitueds, |V_i|^2, and power flows, P_ik and Q_ik, and 2) sum to a multiple of 2\pi radians around every cycle.
They are called the \emph{Distflow} and \emph{Linearized Distflow model} \cite{Low2014d,Baran1989}. Both models are valid when the underlying network topology is a tree, which is indeed the case in this paper (as we consider a line topology).  Moreover, we show that for a line topology the power flow model equations can be rewritten recursively for the voltages. For an overview of other representations of power flow, we refer the reader to \cite[Chapter 2]{Molzahn2019}.

Given the impedance $r$, %$z$,
the voltage at the root node $\tilde{V}_0$ and the power consumptions $\tilde{p}_j, j=1,\ldots,N$, both power flow models satisfy three relations. First, we have power balance at each node: for all $j\in\mathcal{I}\backslash\{0\}$,
\begin{align}
\tilde{S}_{j-1,j}-r\left| I_{j-1,j}\right|^2 = \tilde{s}_j + \tilde{S}_{j,j+1}.\label{eq:power_flow_equations}
\end{align} Here, on the one hand, the quantity $r|I_{j-1,j}|^2$ represents line loss so that $\tilde{S}_{j-1,j}-r|I_{j-1,j}|^2$ is the receiving-end complex power at node $j$ from node $j-1$. On the other hand,  the delivering-end complex power is the sum of the consumed power at node $j$ and the complex power flowing from node $j$ to node $j+1$. Second, by Ohm's law, we have for each edge $\epsilon_{j-1,j} \in \mathcal{E}$,
\begin{align}
\tilde{V}_{j-1} - \tilde{V}_j = rI_{j-1,j}\label{eq:ohm}
\end{align} and third, due to the definition of complex power, we have for each edge $\epsilon_{j-1,j} \in \mathcal{E}$,
\begin{align}
\tilde{S}_{j-1,j} = \tilde{V}_{j-1}\overline{I}_{j-1,j}.\label{eq:conjugate}
\end{align} Decomposing Equation \eqref{eq:power_flow_equations} in real and imaginary parts  leads to Equations \eqref{eq:active_power} and  \eqref{eq:reactive_power}, and rewriting Equations \eqref{eq:ohm} and \eqref{eq:conjugate} leads to Equations \eqref{eq:voltage_magnitude} and \eqref{eq:current_magnitude}, respectively:
\begin{alignat}{2}
\tilde{p}_j & = \tilde{P}_{j-1,j}-r\left|I_{j-1,j}\right|^2-\tilde{P}_{j,j+1}, &\quad j\in\{1,\ldots,N\}, \label{eq:active_power} \\
\tilde{q}_j & = \tilde{Q}_{j-1,j}-\tilde{Q}_{j,j+1}, &\quad j\in\{1,\ldots,N\}, \label{eq:reactive_power} \\
%\tilde{q}_j & = \tilde{Q}_{j-1,j}-x\left|I_{j-1,j}\right|^2-\tilde{Q}_{j,j+1}, \quad j\in\{1,\ldots,N\}, \label{eq:reactive_power} \\
\tilde{V}_j^2 & = \tilde{V}_{j-1}^2-2r\tilde{P}_{j-1,j}+\left|r\right|^2\left|I_{j-1,j}\right|^2, &\quad \epsilon_{j-1,j}\in\mathcal{E}, \label{eq:voltage_magnitude}\\
%\tilde{V}_j^2 & = \tilde{V}_{j-1}^2-2\big(r\tilde{P}_{j-1,j}+x\tilde{Q}_{j-1,j}\big)+\left|r\right|^2\left|I_{j-1,j}\right|^2, \quad \epsilon_{j-1,j}\in\mathcal{E}, \label{eq:voltage_magnitude}\\
\left|I_{j-1,j}\right|^2 & = \frac{\left|\tilde{S}_{j-1,j} \right|^2}{\tilde{V}_{j-1}^2}, &\quad \epsilon_{j-1,j}\in\mathcal{E}, \label{eq:current_magnitude}
\end{alignat} with the understanding that when $j$ is a leaf node $\tilde{P}_{j,j+1}=0$ and $\tilde{Q}_{j,j+1}=0$ in \eqref{eq:active_power} and \eqref{eq:reactive_power}, respectively. In case of the Distflow model, we rewrite Equations \eqref{eq:active_power}--\eqref{eq:current_magnitude} to obtain recursive expressions in terms of the voltages at the nodes, while, in case of the Linearized Distflow model, we simplify Equations \eqref{eq:active_power}--\eqref{eq:current_magnitude} first and then find recursive equations for the voltages.
Furthermore, to avoid a notational issue, we denote the expressions for the different models using the superscripts $D$ (Distflow model equations) and $L$ (Linearized Distflow equations). If there is no superscript specified, the expressions hold for both models.

%Note that $\tilde{V}_j, j=1,\ldots,N$ are dependent on the vector $\tilde{\mathbf{p}}$ and $z$. We write $\tilde{V}_j(\mathbf{p},z)$ when we wish to emphasize this.
The goal in Sections \ref{subsec:lindist} and \ref{subsec:distflow} is to derive equations that allow us to compute if a given power allocation $\tilde{\mathbf{p}}$ is feasible or not under both power flow models (given properties of the distribution network, such as the impedance $r$ % $z$
and base voltage $\tilde{V}_0$). %\red{Furthermore, without loss of generality, we normalize the voltages such that $\tilde{V}_0=1$.}

\subsubsection{Distflow}\label{subsec:distflow}
The first power flow model we discuss is called the \emph{Distflow} model. This model was first proposed in \cite{Baran1989a} and \cite{Baran1989} and derived under the assumption that the underlying network topology is a tree. Further simplifying the network topology into a line and assuming that load nodes consume active power only, direct manipulation and rewriting of Equation \eqref{eq:voltage_magnitude} with the help of Equations \eqref{eq:active_power},\eqref{eq:reactive_power}, and \eqref{eq:current_magnitude} yield a second-order recursive relation between the voltages,
\begin{align}
\tilde{V}_{N-1}^D & = \tilde{V}_N^D + \frac{r\tilde{p}_N}{V_N^D},\label{eq:V_N-1}\\
\tilde{V}_{j-1}^D & = 2\tilde{V}_j^D-\tilde{V}_{j+1}^D+\frac{r\tilde{p}_j}{\tilde{V}_j^D}, \quad j = 1,\ldots,N-1 \label{eq:V_j-1D}
\end{align} %where
%\begin{align}\label{eq:hat_r}
%\hat{r}:=\frac{r^2+x^2}{r+x},
%\end{align}
where the voltage at the root node, $\tilde{V}_0^D$, and the charging rates $\tilde{p}_j$ for $j=1,\ldots,N$ are given. %We can always scale the voltages with respect to $\tilde{V}_0^D$, so we can set $\tilde{V}_0^D=1$.
Relabeling Equations \eqref{eq:V_N-1} and \eqref{eq:V_j-1D} by ${V_j}^D := \tilde{V}_{N-j}^D$ and $p_j:=\tilde{p}_{N-j}$ results in a new recursive relation such that each further term of the sequence is defined as a function of the preceding terms. Notice that in terms of the new variables ${V_N}^D$ and for $j=0,\ldots,N-1$, $p_j$  are given. %\red{To avoid cumbersome notation, we write $V_j^D$ and $p_j$ instead of $\tilde{V_j}^D$ and $\tilde{p_j}$, respectively.}
Thus, we have the following second-order recursive relation between the voltages,
\begin{align}
{V_{1}}^D & = {V_0}^D + \frac{rp_0}{{V_0}^D}, \label{eq:DF1}\\
{V_{j+1}}^D & = 2{V_j}^D-{V_{j-1}}^D+\frac{rp_j}{{V_j}^D}, \quad j = 1,\ldots,N-1 \label{eq:DF2}
\end{align} where ${V_N}^D$ is known.

Given the recurrence relation in \eqref{eq:DF1} and \eqref{eq:DF2}, it is not straightforward to find out if a given power allocation $\mathbf{\tilde{p}}$ is feasible or not. We establish in Lemma \ref{lemma:V_increasing} that the sequence of voltages ${V_j}^D, j=0,\ldots,N$ is an increasing sequence. Thus, computing if a given power allocation is feasible, requires finding a voltage ${V_0}^D$ such that the known voltage ${V_N}^D$ satisfies the maximal voltage drop constraint as in Equation \eqref{eq:voltage_drop_C}.

However, in what follows, we find an equivalence of the recursive relations in \eqref{eq:DF1} and \eqref{eq:DF2} between the situation where the voltage ${V_N}^D$ is known and the situation where the voltage ${V_0}^D$ is known, such that the voltage drop constraint is satisfied. In the latter situation, one can find out easily if a given power allocation $\mathbf{p}$ (with distribution network parameter $r$ and voltage ${V_N}^D$) is feasible: initialize the recursion with ${V_0}^D>0$ and iterate through \eqref{eq:DF1} and \eqref{eq:DF2} and check, using the voltage ${V_N}^D$, if the maximal voltage drop constraint is satisfied.

To prove the desired equivalence, we follow the approach in \cite[Chapter 5]{Vasmel2019}. We first derive an alternative expression for the voltages under the Distflow model as a double summation, instead of the voltages given by Equations \eqref{eq:DF1} and \eqref{eq:DF2}. Then, we show that the voltage at node $j$, for $j=0,\ldots,N$, as a function of the voltage ${V_0}^D$ is an increasing function. The proofs of Lemmas \ref{lemma:alternative_Vn} and \ref{lemma:v_increasing} and Proposition \ref{prop:equivalence} can be found in Appendix \ref{sec:equivalence}.

In Lemma \ref{lemma:alternative_Vn}, we derive an alternative expression for the voltages under the Distflow model.
\begin{lemma}\label{lemma:alternative_Vn} Let ${V_j}^D, j=0,1,\ldots,N$, be as in \eqref{eq:DF1} and \eqref{eq:DF2}. Then,
\begin{enumerate}
\item \begin{align}
{V_{j+1}}^D-{V_j}^D=\sum_{i=0}^j \frac{rp_i}{{V_i}^D},\quad j = 0,\ldots,N-1,\label{eq:alternative_first_sum}
\end{align}
\item \begin{align}
{V_{j}}^D = {V_0}^D + \sum_{n=0}^{j-1} \sum_{i=0}^n \frac{rp_i}{{V_i}^D},\quad j = 0,\ldots,N.\label{eq:alternative_second_sum}
\end{align}
\end{enumerate}
\end{lemma}
Now that we have an alternative expression for the voltages under the Distflow model, we show, in Lemma \ref{lemma:v_increasing} that the voltages under the Distflow as a function of the voltage ${V_0}^D$ are increasing functions. To this end, we let the voltage ${V_0}^D$ be positive and define $v_j: \mathbb{R}_+ \to \mathbb{R}_+$ be the function defined by the equation $v_j({V_0}^D) = {V_j}^D$.
\begin{lemma}\label{lemma:v_increasing}
Let the voltage ${V_0}^D$ be positive $({V_0}^D>0)$. If ${V_N}^D\leq 2{V_0}^D$, then
\begin{align*}
0\leq \frac{dv_j}{d{V_0}^D}\leq 1,
\end{align*} for all $j=0,\ldots,N$.
\end{lemma}
In Lemma \ref{lemma:v_increasing}, we restricted ourselves to a network of nodes where the voltage ${V_N}^D$ can maximally be two times the voltage ${V_0}^D$. However, this is not an additional restriction since the inequality $0<\Delta\leq \frac{1}{2}$ immediately implies that ${V_N}^D\leq 2{V_0}^D$. %means that, $\frac{{V'_N}^D-{V'_0}^D}{{V'_N}^D}\leq \frac{1}{2}$, or in other words, that $\Delta\leq \frac{1}{2}$.

Having Lemma \ref{lemma:v_increasing} at our disposal, we can prove the desired equivalence, which we state in Proposition \ref{prop:equivalence}.
\begin{proposition}\label{prop:equivalence}
Let the resistance $r$ and power consumptions $p_j, j=0,\ldots,N-1$ be given. Then the following equivalence, for $0< \Delta\leq \frac{1}{2}$,  for the voltages (under the Distflow model) at the root node and the node at the end of the line holds:
\begin{multline}\Big(\exists\,x \geq (1-\Delta)c~\text{s.t.}~{V_0}^D = x~\text{and}~v_N({V_0}^D) = c\Big)~\text{if and only if}\\~\Big({V_0}^D = (1-\Delta)c~\text{and}~v_N({V_0}^D) \leq c\Big),\label{eq:equivalence}
\end{multline} where $c>0$.
\end{proposition}

Due to Proposition \ref{prop:equivalence}, we can, given the voltage $V_0^D$, recursively determine the voltage $V_N^D$. The value for the voltage ${V_0}^D$ is arbitrary, so for simplicity, we set ${V_0}^D = 1$ (and thus need to determine if ${V_N}^D\leq \frac{1}{1-\Delta}$). For the rest of the paper, we work with the convention that ${V_0}^D=1$, i.e.,
\begin{align}
{V_{1}}^D & = 1 + rp_0, \label{eq:DF1_1}\\
{V_{j+1}}^D & = 2{V_j}^D-{V_{j-1}}^D+\frac{rp_j}{{V_j}^D}, \quad j = 1,\ldots,N-1 \label{eq:DF2_1}.
\end{align}
Observe that Equations \eqref{eq:DF1_1}--\eqref{eq:DF2_1} are non-linear. For $\epsilon_{j-1,j}\in\mathcal{E}$, we apply the transformation,
\begin{align}
\mathbf{W}(\epsilon_{j-1,j}) = \begin{pmatrix}
{{V_{j-1}}}^2 & V_{j-1}V_j \\
V_jV_{j-1} & {{V_j}}^2
\end{pmatrix}:=
\begin{pmatrix}
W_{j-1,j-1} & W_{j-1,j} \\
W_{j,j-1} & W_{jj}
\end{pmatrix} \label{eq:transformation_W}
\end{align} to \eqref{eq:DF1_1} and \eqref{eq:DF2_1}. This transformation, for $j=1,\ldots,N-1$, leads to the linear Equations \eqref{eq:W2} and \eqref{eq:W_cond} (in terms of $\mathbf{W}(\epsilon_{ij})$), which turn out to be useful in proving Lemma \ref{lemma:compactness} in Section \ref{section:proof_distflow}; i.e.
\begin{align}
{W_{j,j+1}}^D & = 2{W_{j,j}}^D-{W_{j-1,j}}^D+rp_j,\label{eq:W2}
\end{align}
where
\begin{align}
{W_{0,0}}^D=1\ \text{and}\  {W_{0,1}}^D = 1+rp_0\label{eq:W_cond}
\end{align} are the initial values. However, if we want to compute the recursion variables ${W_{j,j+1}}^D, j=1,\ldots,N-1$ explicitly, we need the relationship defined in \eqref{eq:W1}, which is not linear:
\begin{align}
W_{j,j}^D & = \frac{{{W_{j-1,j}}^D}^2}{{W_{j-1,j-1}}^D}. \label{eq:W1}
\end{align} Again, we want to keep the maximal voltage drop under control. In this case, the constraint reduces to
\begin{align*}
W_{N,N}^D\leq \left(\frac{1}{1-\Delta} \right)^2,
\end{align*} where $0<\Delta\leq \frac{1}{2}$.
\subsubsection{Linearized Distflow}\label{subsec:lindist}
The second power flow model we consider is the \emph{Linearized Distflow} model. While the first power flow model eventually results in the non-linear Equation \eqref{eq:W1} when wishing to compute the voltage at each node, the Linearized Distflow model has a simpler representation. In the Linearized Distflow model, we assume that the active and reactive power losses $r|I_{j-1,j}|^2$ and $x|I_{j-1,j}|^2$, respectively, are much smaller than the active and reactive power flows $\tilde{P}_{j-1,j}$ and $\tilde{Q}_{j-1,j}$; i.e., the Linearized Distflow approximation neglects the loss terms associated with the squared current magnitudes $|I_{j-1,j}|^2$. Relabeling such that each further term of the sequence is defined as a function of the preceding terms and
rewriting Equation \eqref{eq:voltage_magnitude} using Equations \eqref{eq:active_power},\eqref{eq:reactive_power} and \eqref{eq:current_magnitude}, yields a linear relationship for the squared voltages at each node,
\begin{align}
({{V'_{j+1}}^L})^2 - ({{V'_{j}}^L})^2& =   2r\sum_{m=0}^{N-1-j} p_m,\quad j=1,\ldots,N-1. \label{eq:u_lindist2}
\end{align} where $({{V'_N}^L})^2$ is given. Notice that we do not need an equivalence as in the case of the voltages under the Distflow model since we have an explicit expression for ${V'_0}^L$. In this case, the voltage ${V'_0}^L$ can be computed directly; iteratively using \eqref{eq:u_lindist2}, %setting \eqref{eq:u_lindist2} such that the base voltage is equal to 1 and dropping the accent notation,
we find a closed form solution for the squared voltage ${({{V'_{0}}^L})}^2$, i.e.,
\begin{align}
{({{V'_{0}}^L})}^2 = {({V'_N}^L)}^2-2r\sum_{j=0}^{N-1}\sum_{m=0}^{N-1-j} p_m. \label{eq:constraints}
\end{align}
%& = \left|V_{0}\right|^2-2\sum_{j=1}^{n}\big(r_{j-1,j}P_{j-1,j}\big)\\
Thus, we may rewrite \eqref{eq:constraints} by defining ${W_{j,j}}^L:=({{V'_j}^L})^2$, for $j=0,\ldots,N$. Then, Equation \eqref{eq:constraints} has the linear expression,
\begin{align}
{W_{0,0}}^L = {W_{N,N}}^L - 2r\sum_{j=0}^{N-1}\sum_{m=0}^{N-1-j} p_m.\label{eq:W_lindist}
\end{align} %Similarly as in the case of the Distflow model, a power allocation is feasible under the Linearized Distflow if
%\begin{align*}
%{W_{N,N}}^L\leq \left(\frac{1}{1-\Delta} \right)^2,
%\end{align*} when $W_{0,0}^L = 1$.

\subsubsection{Summary}\label{subsubsec:power_flow_summary} Again, note that $V_j$ and $W_{j,j}, j=1,\ldots,N$ are dependent on the vector $\mathbf{p}$ and resistance $r$. We write $V_j(\mathbf{p},z)$ and $W_{j,j}(\mathbf{p},z)$ as a continuous function of the power allocation $\mathbf{p}$ and distribution network parameter $r$, respectively, when we wish to emphasize this. Finally, we can define our constraint set $\mathcal{C}$. Recall that the only constraint we put on the power allocation is the voltage drop constraint:
\begin{align*}
{W_{N,N}}-W_{0,0}\leq \left(\frac{1}{1-\Delta} \right)^2 -1,
\end{align*} where $0<\Delta\leq \frac{1}{2}$, on both power flow models. Thus, we define the constraint set as follows,
\begin{align}
\mathcal{C} & := \left\{\mathbf{p}: {W_{N,N}}-{W_{0,0}}\leq \left(\frac{1}{1-\Delta} \right)^2-1,\quad 0<\Delta\leq\frac{1}{2}\right\}.\label{eq:set_of_constraints}
\end{align}

%Without loss of generality, we normalize Equations \eqref{eq:DF1} and \eqref{eq:DF2}, \eqref{eq:W2} and \eqref{eq:W_cond} for the Distflow model, and \eqref{eq:W_lindist} for the Linearized Distflow model. Recall that in our problem the power allocation $\mathbf{p}$ and base voltages ${V'_0}^D$ and ${V'_0}^L$ are known and the goal is to compute the maximum feasible arrival rate such that the voltage drop is not high. \red{introduction} Therefore, scale \eqref{eq:DF1} and \eqref{eq:DF2} by ${V'_0}^D$ and define $V_j^D = \frac{{V'_j}^D}{{V'_0}^D}$, such that
%\begin{align}
%V_0^D = 1\ \text{and}\ V_1^D = 1+\left(\frac{1}{{V'_0}^2}\right)\hat{r}p_0,\label{eq:scaled_voltages_D_init}
%\end{align} and
%\begin{align}
%V_{j+1}^D = 2V_j^D-V_{j-1}^D + \frac{1}{({V'_0}^D)^2}\frac{\hat{r}p_j}{V_j^D}.\label{eq:scaled_voltages_D}
%\end{align} and the voltage drop constraint reduces to
%\begin{align}
%V_N^D = \frac{{V'_N}^D}{{V'_0}^D} \leq \frac{1}{1-\Delta}.\label{eq:constraint_scaled_voltages_D}
%\end{align} Similarly, we get,
%\begin{align}
%W_{0,0}^D = 1\ \text{and}\ W_{0,1}^D = \left(\frac{1}{({{V'_0}^D})^2}\right)\hat{r}p_0,
%\end{align} and
%\begin{align}
%W_{j,j+1}^D = 2W_{j,j}^D-W_{j,j-1}^D + \left(\frac{1}{({{V'_0}^D})^2}\right)\hat{r}p_j.
%\end{align}
In summary, the main features of the two power flow models we consider are given in the table below.

\begin{table}[h!]
    \centering
    %\caption{MWE of a table using booktabs}
    %\rowcolor{5}{}{gray!10}
    \begin{tabular}{llllll}
        \toprule
        \rowcolor{gray!10} &&& \multicolumn{2}{c}{Load flow models} \\
        \cmidrule(lr){4-5}
        && & Linearized Distflow & Distflow \\
        \midrule Given $r$ and $\mathbf{p}$: \\ \\ \midrule
        \rowcolor{gray!10} Original equation & &  &        &  \\ \midrule
        \emph{Initial conditions} & & & $({V_N^L})^2 = \left(\frac{1}{1-\Delta}\right)^2$       & $V_0^D = 1$         \\	
        & & & & $V_1^D = 1+rp_0$ \\	
        \emph{Recursion} & & &  $({V_j^L})^2 = ({V_{j-1}^L})^2+2r\sum_{m=0}^{N-1-j} p_m$      & $V_{j+1}^D = 2V_j^D-V_{j-1}^D+\frac{rp_j}{V_j^D}$  \\ \midrule
        \rowcolor{gray!10} After transformation &&  &  &         \\ \midrule
        \emph{Initial conditions} & & & ${W_{N,N}^L} = \left(\frac{1}{1-\Delta}\right)^2$ & $W_{0,0}^D = 1$ \\
        & & & & $W_{0,1}^D = 1+rp_0$\\
        \emph{Recursion} & & & ${W_{0,0}^L} = W_{N,N}^L - 2r\sum_{j=0}^{N-1}\sum_{m=0}^{N-1-j} p_m$ &        $W_{j,j}^D = \frac{{W_{j-1,j}^D}^2}{W_{j-1,j-1}^D}$ \\
        && &  &   \\
        && &  & $W_{j,j+1}^D = 2W_{j,j}^D-W_{j-1,j}^D+rp_j$\\ \midrule \emph{Voltage drop constraint} & &  & $W_{0,0}^L \geq 1$& $W_{N,N}^D \leq \left(\frac{1}{1-\Delta} \right)^2$ \\
        \bottomrule
    \end{tabular}
    \caption{Summary of the recursions for the Linearized Distflow and the Distflow models.}
    \label{tab:load_flow_models}
\end{table}

\section{Main results}\label{sec:main_results}
In this section, we present our main results which concern the comparison of the stability of the queuing model under distribution network constraints. The stability of the system is defined as positive recurrence of the Markov process $\mathbf{X}$.

For convenience, we recall the main features of the queuing model. There are $N$ charging stations where customers can charge their EVs. These charging stations are represented by nodes on a line and each line is characterized by the impedance $r$. Each charging station has its own arrival stream of vehicles with homogeneous rate $\lambda$ and charging rate $p_i$, for $i=1,\ldots,N$. The way the powers $p_i$ are allocated among charging stations is given by the mechanism as described in Section \ref{subsec:queueing} and are constrained by the constraint set $\mathcal{C}$ as defined in  \eqref{eq:set_of_constraints}.

Within this framework, we compute, in Theorems \ref{THM:LINDIST} and \ref{THM:DISTFLOW}, the maximal feasible arrival rates when the number of charging stations, denoted by $N$, goes to infinity, such that the queuing model is stable given distribution network constraints. In other words, we compute the arrival rates such that the maximal voltage drop is attained. Then, in Section \ref{SUBSEC:COMPARISON}, we compare these arrival rates explicitly.
%Indeed, if $W_{N,N} = V_N^2\leq 1+\delta$, then $V_N\leq \sqrt{1+\delta}$. %In this section, we show the asymptotic behavior of voltages under the Linearized Distflow and Distflow model in the queueing model of EV-charging introduced in Section \ref{sec:model_description}.  %So, if one needs to satisfy the voltage drop constraint $V_N-V_0\leq \alpha$, one can set $\delta = \alpha(\alpha+2V_0)$.

Our work is based on \cite{Aveklouris2019} and \cite{Shneer2018}, and both provide a way to prove stability. So before we continue with the main results on stability in Sections \ref{subsec:lindist} and \ref{subsec:distflow}, we discuss these two ways to prove stability. First, we start with a brief discussion of the approach in \cite{Aveklouris2019} -- the method of convex relaxation. Then, we  study an extension of \cite[Theorem 11]{Shneer2018} that we use to prove Theorems \ref{THM:LINDIST} and \ref{THM:DISTFLOW} and provides stability of the Markov process $\mathbf{X}$ as well.

The approach of convex relaxation in \cite{Aveklouris2019} goes as follows. It is known that the utility optimization problem in \eqref{eq:argmax_p} subject to the constraint set defined in \eqref{eq:set_of_constraints} is in general non-convex. However, since our underlying network topology is a line, we obtain that the convex-relaxation of \eqref{eq:argmax_p} under the constraint defined in \eqref{eq:set_of_constraints} is exact \cite{Aveklouris2019}. Then, for all $\alpha$-fair allocations, the Markov process $\mathbf{X}$ is stable \cite{Bonald2006a}.

The approach in \cite{Shneer2018} is more general in the sense that we do not need $\alpha$-fair algorithms to prove stability. It holds for a larger class of functions. If the constraint set $\mathcal{C}$ is compact and coordinate-convex, then the Markov process $\mathbf{X}$, which represents the number of EVs charging at every station, is stable if there exists a vector $\boldsymbol{\nu}\in\mathcal{C}$ such that $\boldsymbol{\lambda}<\boldsymbol{\nu}$ according to \cite{Shneer2018}. In fact, in Appendix \ref{sec:stability_result}, we show that we do not need coordinate-convexity, but only compactness of the constraint set $\mathcal{C}$ to prove stability.

\subsection{Stability conditions under the Linearized Distflow model}\label{subsec:lindist}

In this section, we show the computation of the specific arrival rate under the Linearized Distflow model such that the process is stable. That is, we show there exists an explicit threshold $\lambda_N^L$ that only depends on the distribution network parameter $r$, the maximal voltage drop parameter $\Delta$ and the number of charging stations $N$ in the network such that for all arrival rates below this threshold, the queuing model is stable.

%Again, consider the EV-charging model. There are $N$ charging stations where customers can charge their EVs. These charging stations are represented by nodes on a line. Each charging station has its own arrival stream of vehicles with homogeneous rate $\lambda$ and charging rate $p_i$, for $i=1,\ldots,N$. Furthermore, not all power allocations are feasible due to the voltage drop constraint. Now we are ready to formulate Theorem \ref{thm:lindist}.

 %Then, from the stability of the process $\mathbf{X}$, we derive stability conditions on the homogeneous arrival rate $\lambda$, when the number of charging stations $N$ is growing to infinity, such that the voltage drop $W_{N,N}-W_{0,0}$ is not too high, i.e., $W_{N,N}-W_{0,0}\leq\delta$.

\begin{theorem}
Consider the queuing model in Section \ref{subsec:queueing}. The Markov process $\mathbf{X}$ is positive recurrent, if
\begin{align}
\lambda < \lambda_N^L := \frac{\frac{1}{r}\left(\left(\frac{1}{1-\Delta}\right)^2-1 \right)}{N(N+1)}.\label{eq:lambda_N^L}
\end{align}
\label{THM:LINDIST}
\end{theorem}

The proof is given in Section \ref{section:proof_lindist}. It exploits the explicit expression of the squared voltage ${W_{0,0}^L}^2$ and uses \cite[Theorem 11]{Shneer2018}. We show for the arrival rate $\lambda_N^L$ that the maximal allowed voltage drop is attained, so that $\boldsymbol{\lambda}_N^L$ is still in the constraint set $\mathcal{C}$.
% and that $\lambda_N^L$ is contained in the constraint set $\mathcal{C}$. Then, we simply apply Theorem \ref{thm:stability} of Appendix \ref{appendix:stability}.
\begin{remark}\label{remark:lambda_D}
Under the Linearized Distflow approximation, defined in \eqref{eq:W_lindist}, %all feasible arrival rates are given by the set $\bigg\{\lambda: \lambda < \frac{\left(\frac{1}{1-\Delta} \right)^2-1}{rN(N+1)} \bigg\}$, so that
the maximum feasible arrival rate is decaying as the inverse of the square of the number of nodes.
\end{remark}

We now proceed with the stability result under the Distflow model.

\subsection{Stability conditions under the Distflow model}\label{subsec:distflow}
In this section, we show there exists a specific arrival rate $\lambda_N^D$ that only depend on the resistance $r$, the maximal voltage drop parameter $\Delta$ and the number of charging stations $N$ in the network such that for every arrival rate below this threshold the system is stable. Furthermore, as the number of charging stations $N$ goes to infinity, we show the convergence of a suitably scaled version of $\lambda_N^D$ to a \emph{critical} arrival rate $\lambda_c^D$.

\begin{theorem}
Consider the queuing model in Section \ref{subsec:queueing}. There exists $\lambda_N^D$ such that the Markov process $\mathbf{X}$ is positive recurrent, if
\begin{align*}
\lambda < \lambda_N^D.
\end{align*} Moreover, $N^2\lambda_N^D \to \lambda_c^D$ as $N\to\infty$ with
\begin{align}
\lambda_c^D = \frac{\pi}{2r}\text{erfi}^2\left(\sqrt{\ln\left(\frac{1}{1-\Delta}\right)} \right).\label{eq:critical_lambda}
\end{align} %Under the Distflow approximation, defined in \eqref{eq:DF1} and \eqref{eq:DF2}, the maximum feasible arrival rate is decaying as the logarithm of the number of nodes times the square of the increasing number of nodes, i.e. $\lambda \sim \frac{\delta}{N^2\ln(N)}$.
\label{THM:DISTFLOW}
\end{theorem}

The imaginary error function $\text{erfi}(z)$ is the function defined by
%\begin{align}
%\text{erfi}(z) = -\mathrm{i}\ \text{erf}(\mathrm{i}z),
%\end{align} where $\text{erf}(w) = \frac{2}{\sqrt{\pi}}\int_0^w \exp(-v^2)dv$ is the well-known error function, or otherwise by
\begin{align*}
\text{erfi}(z) = \frac{2}{\sqrt{\pi}}\int_0^z \exp(v^2)dv = \frac{1}{\mathrm{i}}\text{erf}(\mathrm{i}z),
\end{align*} where $\text{erf}(x) = \frac{2}{\sqrt{\pi}}\int_0^x \exp(-t^2)dt$ is the standard error function.

The proof is given Section \ref{section:proof_distflow}. Unlike the proof of Theorem \ref{THM:LINDIST}, we do not have an explicit expression for the voltage $V_N^D$. Thus, we require a different approach to find the maximal feasible arrival rate $\lambda_N^D$. Instead of solving for the value $\lambda_N^D$ such that the maximal allowed voltage drop is attained, as in \eqref{eq:set_of_constraints}, we approximate the voltage $V_N^D$ by a continuous function. Then, we show that this continuous function converges to the voltage $V_N^D$ as $N$ goes to infinity and compute the arrival rate such that the maximal voltage drop is attained under this new approximation.

\subsection{Comparison of stability regions}\label{SUBSEC:COMPARISON}
In this section, we compare the critical arrival rates under both power flow models. The expression for the critical arrival rate under the Distflow model is given in \eqref{eq:critical_lambda}. For the Linearized Distflow model, we see from \eqref{eq:lambda_N^L} that $N^2\lambda_N^L\to\lambda_c^L$ as $N\to\infty$ with
\begin{align}
\lambda_c^L = \frac{1}{r}\left(\left(\frac{1}{1-\Delta} \right)^2-1 \right).\label{eq:critical_lambda_L}
\end{align}
%Recall $\hat{r} = \frac{r^2+x^2}{r+x}$. For convenience, define $\kappa:=\frac{x}{r}$.
Then, we have for the ratio between $\lambda_c^D$ and $\lambda_c^L$,
\begin{align}
\frac{\lambda_c^D}{\lambda_c^L} & = \frac{2(1-\Delta)^2\left(\int_0^{\sqrt{\ln\left(\frac{1}{1-\Delta}\right)}}\exp(u^2)du \right)^2}{\Delta(2-\Delta)} \nonumber\\
%& = \frac{2(1-\Delta)^2\left(\int_0^{\sqrt{\ln\left(\frac{1}{1-\Delta}\right)}}\exp(u^2)du \right)^2}{\Delta(2-\Delta)} \nonumber \\
& =: P(\Delta).\label{eq:P(delta)}
\end{align} %where
%\begin{align}
%P(\Delta) := \frac{2(1-\Delta)^2\left(\int_0^{\sqrt{\ln\left(\frac{1}{1-\Delta} \right)}}\exp(u^2)~du \right)^2}{\Delta(2-\Delta)}.\label{eq:P(delta)}
%\end{align}
We study the behavior of the function $P(\Delta)$ in Theorem \ref{THM:DECREASINGNESS_P}.
\begin{theorem}\label{THM:DECREASINGNESS_P} Let $P(\Delta)$ be defined as in \eqref{eq:P(delta)} for $0<\Delta\leq\frac{1}{2}$. Then $P(\Delta)$ is a strictly decreasing function from 1 at $\Delta=0$ to $\frac{\pi}{6}\text{erfi}\left(\sqrt{\ln(2)} \right)^2\approx 0.77$ at $\Delta =\frac{1}{2}$.
\end{theorem}
The proof of Theorem \ref{THM:DECREASINGNESS_P} can be found in Appendix \ref{sec:appendix_comparison}. The importance of Theorem \ref{THM:DECREASINGNESS_P} lies in the fact that it implies that the critical arrival rate under the Distflow is always smaller than the critical arrival rate under the Linearized Distflow model and that we are able to quantify the difference between the critical arrival rates.% and note that the global maximum of $\frac{1+\kappa}{1+\kappa^2}$ is equal to $\frac{1}{2}+\frac{1}{\sqrt{2}}$. Then, to make a qualitative statement about the size of \eqref{eq:fraction_lambdas}, we let $\Delta_0$ be such that
For several values of $\Delta$, Table \ref{tab:kappa_bounds} shows %threshold values for $\kappa$ as given in Theorem \ref{THM:STAB_REGION}. For two specific values of $\Delta$, i.e., $\Delta = 0.05$ and $\Delta = 0.1$
the ratio between $\lambda_c^D$ and $\lambda_c^L$.

\begin{table}[h!]
\centering
 \begin{tabular}{c c }
 \hline
 \rowcolor{gray!10} $\Delta$ & $\lambda_c^D/\lambda_c^L$   \\ [1ex]
 \hline
 0.01 & 0.9966    \\
 0.05 & 0.9828   \\
 0.1 & 0.9647   \\
 0.2 & 0.9248  \\ [1ex]
 \hline
 \end{tabular}
 \caption{The fraction $\lambda_c^D/\lambda_c^L$ for specific values of $\Delta$.}
 \label{tab:kappa_bounds}
\end{table}
% there exists an interval for really small values of $\kappa$ where $\lambda_c^D< \lambda_c^L$, an interval within the range [0,1] where $\lambda_c^D\geq \lambda_c^L$ and a large region that indicates $\lambda_c^D< \lambda_c^L$ for values of $\kappa$ greater than 1.
\begin{figure}[h!]
\centering
\includegraphics[scale=0.5]{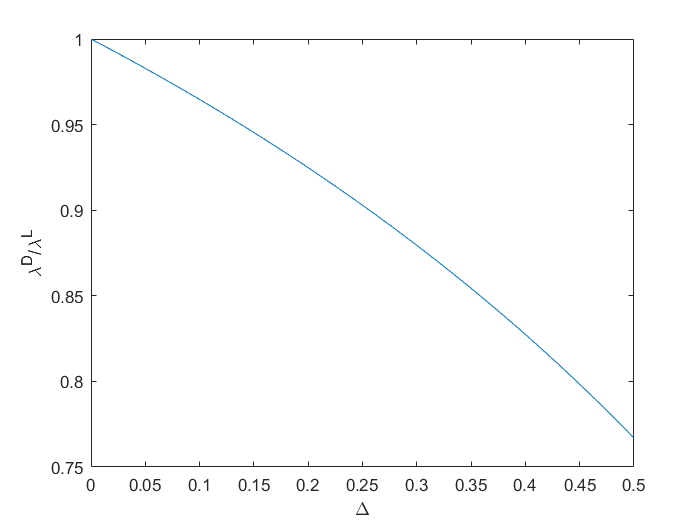}
\caption{The fraction $\frac{\lambda_c^D}{\lambda_c^L}$ via \eqref{eq:P(delta)} for $0<\Delta\leq \frac{1}{2}$.}
\label{fig:kappa_bounds}
\end{figure}
Similar to Table \ref{tab:kappa_bounds}, one can see from Figure \ref{fig:kappa_bounds} that the ratio between $\lambda_c^D$ and $\lambda_c^L$ decreases as $\Delta$ increases. %that the threshold values that we found in Theorem \ref{THM:STAB_REGION} are in the range [0,1] and in the case where $\kappa$ is relatively high, we certainly have $\lambda_c^D< \lambda_c^L$.
When we allow a voltage drop of $50\%$, i.e., $\Delta = \frac{1}{2}$, the maximal feasible arrival rate under the Linearized Distflow model is $\approx 20\%$ higher than under the Distflow model. In more realistic situations, when $\Delta$ is close to 0, the difference between these critical arrival rates is even smaller.  %fluctuates around 1. For example, if $\kappa = 1/7$ and $\Delta = 0.05$, $\frac{\lambda_c^D}{\lambda_c^L} = 1.1$. Thus, according to the Distflow model, the maximal feasible arrival rate is $10\%$ higher than under the Linearized Distflow model. However, there are certainly reasonable situations, when $\kappa$ is close to 1, where the maximal feasible arrival rate under the Distflow model is smaller than the maximal feasible arrival rate under the Linearized Distflow model, but the difference is small.

\section{Proof of Theorem \ref{THM:LINDIST}}\label{section:proof_lindist}

The proof consists of three steps. First, we show that our queuing model description, as in Section \ref{subsec:queueing}, is contained in the general model framework that is given by \cite{Shneer2019c}. Then, in order to apply their stability result, which we discuss below, we show that the constraint set $\mathcal{C}$ is compact. Last, we show that the maximal feasible arrival rate $\lambda_N^L$ is given by Equation \eqref{eq:lambda_N^L}.

\subsection{Queuing model framework}\label{subsec:correct_model}
We consider the general model framework in \cite{Shneer2018} and compare it with the continuous-time model of Section \ref{sec:model_description}. In their setting, there is an arbitrary network of queues such that individual instantaneous service rates may depend on the state of the entire system. That corresponds to our network of charging stations such that the power allocation depends on the total number of cars charging at the same time. The network state is represented by a set $\mathbf{X}(t)=(X_i(t),i\in N)$ of the queue lengths $X_i$ at the network nodes $i\in N$ at time $t$. In our case, the network state is represented by the number of cars charging at each charging station. Arrivals into queue $i$ occur according to a Poisson process with a constant rate $\lambda$, %smaller than 1 is not necessary, see email seva
independent of all other processes. The instantaneous service rate $\mu_i$ of a node represents the intensity of the Poisson process modeling departures (service completions) of the node. In their case, the service rate allocation algorithm $\boldsymbol{\psi}(\mathbf{X}(t))$, mapping a network state $\mathbf{X}(t)$ into a set of instantaneous service rates $\boldsymbol{\mu}$, is all that is needed to specify the service allocation algorithm. In our case, the power allocation is given by a constrained optimization problem, mapping the network state, the number of EVs at each charging station, into a power allocation.

If we are in this general model framework and there exists $\boldsymbol{\nu}\in\mathcal{C}$ such that $\boldsymbol{\lambda}<\boldsymbol{\nu}$ and the set $\mathcal{C}$ is compact and coordinate-convex (i.e. if a vector $\mu$ belongs to $\mathcal{C}$ and $\mu'\leq \mu$ coordinate-wise, then $\mu'\in\mathcal{C}$), then the Markov process $\mathbf{X}$ is positive recurrent according to \cite[Theorem~11]{Shneer2018}. In Appendix \ref{sec:stability_result}, we show that the set $\mathcal{C}$ does not need to be necessarily coordinate-convex, since we can construct a coordinate-convex set $\mathcal{C}_1$ such that the following are equivalent:
\begin{align*}
\text{there exists}\ \boldsymbol{\nu}\in\mathcal{C}\ \text{such that}\ \boldsymbol{\lambda}<\boldsymbol{\nu},
\end{align*} and
\begin{align*}
\text{there exists}\ \boldsymbol{\nu}\in\mathcal{C}_1\ \text{such that}\ \boldsymbol{\lambda}<\boldsymbol{\nu}.
\end{align*} We are thus left to show that there exists $\boldsymbol{\nu}\in\mathcal{C}$ such that $\boldsymbol{\lambda}<\boldsymbol{\nu}$ and that $\mathcal{C}$ is a compact set. First, we show the compactness of the set $\mathcal{C}$.

\subsection{Compactness of the constraint set $\mathcal{C}$}

In this section, we show that the constraint set $\mathcal{C}$ is compact, i.e. we show that the set is closed and bounded.
\begin{lemma}\label{lemma:compactness_lindist}
The constraint set $\mathcal{C}$, defined in \eqref{eq:set_of_constraints}, is compact for the Linearized Distflow model.
\end{lemma}
\begin{proof}
Since $W_{0,0}^L$ is a continuous function, the constraint set $\mathcal{C}$ is closed.
%\begin{align}\left\{\mathbf{p}: W_{N,N}^L(\mathbf{p},z)\leq \left(\frac{1}{1-\Delta}\right)^2\right\}
%\end{align} is closed.
Furthermore, by definition of the constraint set in \eqref{eq:set_of_constraints} and \eqref{eq:W_cond}, we have
\begin{align*}
{W_{0,0}^L}^2 & = {W_{N,N}^L}^2-2r\sum_{j=0}^{N-1}\sum_{m=0}^{N-1-j} p_m\\
& = \left(\frac{1}{1-\Delta} \right)^2-2r\sum_{j=0}^{N-1}\sum_{m=0}^{N-1-j} p_m \geq 1.
\end{align*}
%\begin{align*}
%\left(\frac{1}{1-\Delta} \right)^2 \geq W_{N,N}^L = 1 + 2r\sum_{j=0}^{N-1}%\sum_{m=0}^{N-1-j} p_m.%
%\end{align*}
Thus,
\begin{align*}
\sum_{j=0}^{N-1}\sum_{m=0}^{N-1-j} p_m \leq \frac{1}{2r}\left(\left(\frac{1}{1-\Delta} \right)^2-1\right).
\end{align*} Therefore, $\mathbf{p}$ is bounded. Hence, the set of constraints $\mathcal{C}$ is bounded and closed, which implies compactness.
\end{proof}
Next, we show there exists $\boldsymbol{\nu}\in\mathcal{C}$ such that $\boldsymbol{\lambda}<\boldsymbol{\nu}$.
\subsection{Computation of $\lambda_N^L$ such that maximal voltage drop is attained}

In this section, we compute the maximal feasible arrival rate $\lambda_N^L$ such that $\boldsymbol{\lambda}_N^L = (\lambda_N^L,\ldots,\lambda_N^L)\in \mathcal{C}$ (where we define $\lambda_N^L$ below). Then, by \cite[Theorem 11]{Shneer2018}, for all $\lambda<\lambda_N^L$, the Markov process $\mathbf{X}$ is positive recurrent. First, we make the following remark. We do not necessarily need that all arrival rates for all nodes are the same. It is sufficient that $\boldsymbol{\lambda} < \boldsymbol{\lambda}_N^L$ for some arrival rate vector $\boldsymbol{\lambda}$. However, to %keep the analysis in Section \ref{section:proof_distflow} tractable and to
make a comparison between the two power flow models, we let all arrival rates for all nodes be the same, i.e.; $\boldsymbol{\lambda} = (\lambda,\ldots,\lambda)$. Now, we establish that $\boldsymbol{\lambda}_N^L$ is contained in the constraint set $\mathcal{C}$. By definition, we have
\begin{align*}
\lambda_N^L = \frac{\left(\frac{1}{1-\Delta} \right)^2-1}{rN(N+1)}.
\end{align*} Using this arrival rate in the expression for the squared voltage $W_{0,0}^L(\boldsymbol{\lambda}_N^L)$ (cf. Table \ref{tab:load_flow_models}), yields
\begin{align*}
W_{0,0}^L(\boldsymbol{\lambda}_N^L) & = W_{N,N}^L-2r\sum_{j=0}^{N-1}\sum_{m=0}^{N-1-j} \lambda_N^L \\
& = \left(\frac{1}{1-\Delta} \right)^2-2r\lambda_N^L \frac{1}{2}N(N+1)\\
& = 1.
\end{align*}
Hence, the vector $\boldsymbol{\lambda}_N^L = (\lambda_N^L,\ldots,\lambda_N^L)$ is contained in the constraint set $\mathcal{C}$ and for $\lambda<\lambda_N^L$ the Markov process $\mathbf{X}$ is positive recurrent.

\section{Proof of Theorem \ref{THM:DISTFLOW}}\label{section:proof_distflow}
The proof consists of four steps. Similarly, as in the proof of Theorem \ref{THM:LINDIST}, we show that our queuing model description, as in Section \ref{subsec:queueing}, is contained in the general model framework that is given by \cite{Shneer2018} (see Appendix \ref{subsection:model}) and that the constraint set $\mathcal{C}$ is compact for the Distflow model. Then, we need an intermediate step before we can compute the arrival rate $\lambda_N^D$ such that the maximal voltage drop is attained. Unlike the proof of Theorem \ref{THM:LINDIST}, we do not have an explicit expression for the voltage $V_N^D$. Therefore, our next step is to approximate the voltage $V_N^D$ by a continuous function, as $N\to\infty$. In Appendix \ref{sec:conv_VN_to_V(1)}, we show the convergence of the approximation of the voltage to the actual value $V_N^D$, as $N\to\infty$. This approximation allows us to compute the arrival rate $\lambda_N^D$ such that the maximal voltage drop is attained, which concludes the proof.

\subsection{Queuing model framework}
Again, we consider the continuous-time model as in Section \ref{sec:model_description} and we follow the exact same reasoning as in Section \ref{subsec:correct_model}. The only difference in this description is given by the constrained optimization problem. In this case, the voltages are computed according to the Distflow model instead of the Linearized Distflow model.

\subsection{Compactness of the constraint set $\mathcal{C}$}
To establish compactness in case of the Distflow model, we use the fact that the voltages $V_0^D,\ldots,V_N^D$ form an increasing sequence. This is proven in Lemma \ref{lemma:V_increasing}.
\begin{lemma}\label{lemma:V_increasing}
The voltages $V_0^D,\ldots,V_N^D$ are an increasing sequence.
\label{lemma:increasing_seq}
\end{lemma}
\begin{proof}
We give a proof by induction on $j$. Recall that $V_0^D=1$. Then from \eqref{eq:DF1_1}, we have $V_1^D = 1+rp_0\geq 1 = V_0^D$ and for $j=1$, we have from \eqref{eq:DF2_1}
\begin{align*}
V_2^D & = 2V_1^D-V_0^D+\frac{rp_1}{V_1^D}
= (1+rp_0)+rp_0+\frac{rp_1}{1+rp_0}
\geq 1+rp_0 = V_1^D.
\end{align*}
Suppose that $V_j^D \geq V_{j-1}^D$ for some $j\in \{2,\ldots,n-1\}$. Then,  by \eqref{eq:DF2_1} and the induction hypothesis we have
\begin{align*}
V_{j+1}^D &\geq V_j^D + V_{j-1}^D-V_{j-1}^D + \frac{rp_j}{V_j^D}
= V_j^D + \frac{rp_j}{V_j^D} = V_j^D\left(1+\frac{rp_j}{{V_j^D}^2}\right)\geq V_j^D
\end{align*} since $r\geq 0,p_j\geq0$ and $V_j^D\geq V_0^D=1$.
\end{proof}
In fact, it is easier to work with the squared voltages to prove the compactness of the constraint set $\mathcal{C}$ for the Distflow model. Therefore, we present the analogous result in Corollary \ref{cor:V_increasing}.
\begin{corollary}\label{cor:V_increasing}
The squared voltages $W_{0,0}^D, W_{0,1}^D, W_{1,1}^D,\ldots,W_{N-1,N}^D,W_{N,N}^D$ are an increasing sequence.
\end{corollary}
\begin{proof}
This is an immediate consequence of Lemma \ref{lemma:V_increasing}. Since, $V_j^D\geq 1$, for all $j=0,\ldots,N$, we have
\begin{align*}
W_{j,j}^D = V_{j}^DV_{j}^D \geq V_{j-1}^DV_{j}^D = W_{j-1,j}^D
\end{align*} for all $j=0,\ldots,N$.
\end{proof}
Now, we are in position to prove Lemma \ref{lemma:compactness_distflow}.
\begin{lemma}\label{lemma:compactness_distflow}
The constraint set $\mathcal{C}$, defined in \eqref{eq:set_of_constraints}, is compact for the Distflow model.\label{lemma:compactness}
\end{lemma}
\begin{proof}
Since $W_{N,N}^D(\mathbf{p},z)$ is continuous, the set $\mathcal{C}$ is closed.
By definition, we have
\begin{align*}
\left(\frac{1}{1-\Delta} \right)^2 \geq W_{N,N}^D(\mathbf{p},z).
\end{align*} Then, we show a lower bound on the squared voltage $W_{N,N}^D$.  By Corollary \ref{cor:V_increasing} and Equation \eqref{eq:W2}, we have
\begin{align}
{W_{N,N}^D}({\mathbf{p}},z) & \geq {W_{N-1,N}^D}({\mathbf{p}},z) \nonumber \\
& = 2{W_{N-1,N-1}^D}({\mathbf{p}},z)-{W_{N-2,N-1}^D}({\mathbf{p}},z)+rp_{N-1} \nonumber \\
& \geq 2{W_{N-2,N-1}^D}({\mathbf{p}},z)-{W_{N-2,N-1}^D}({\mathbf{p}},z)+rp_{N-1}\nonumber \\
& = {W_{N-2,N-1}^D}({\mathbf{p}},z) +rp_{N-1}. \label{eq:inequalityWNN}
\end{align} Now, apply the definition of $W_{N-2,N-1}^D(\mathbf{p},z)$ in Equation \eqref{eq:W2} and Corollary \ref{cor:V_increasing} to Equation \eqref{eq:inequalityWNN} to find the inequality
\begin{align}
{W_{N-2,N-1}^D}({\mathbf{p}},z) \geq {W_{N-3,N-2}^D}({\mathbf{p}},z)+rp_{N-2}.\label{eq:inequalityWNN2}
\end{align} Inserting \eqref{eq:inequalityWNN2} in \eqref{eq:inequalityWNN} gives
\begin{align*}
{W_{N,N}^D}({\mathbf{p}},z) & \geq {W_{N-3,N-2}^D}({\mathbf{p}},z)+rp_{N-2}+rp_{N-1}.
\end{align*} Applying the definition of ${W_{j-1,j}^D}({\mathbf{p}},z)$ and Corollary \ref{cor:V_increasing} over and over again, for $j=1,\ldots,N-3$ in reversed order, results in the inequality
\begin{align*}
{W_{N,N}^D}({\mathbf{p}},z) & \geq 1+r\left(p_0+\cdots+p_{N-1} \right).
\end{align*} Thus,
\begin{align*}
\left(p_0+\cdots+p_{N-1} \right) \leq \frac{1}{r}\left(\frac{1}{1-\Delta} \right)^2.
\end{align*} Therefore, $\mathbf{p}$ is bounded. Hence, the set of constraints $\mathcal{C}$ is bounded and closed, which implies compactness.
\end{proof}

\subsection{Approximation of $V_N^D$ by continuous counterpart}\label{subsec:approximation_continuous counterpart}
Unlike the proof of Theorem \ref{THM:LINDIST}, we do not have an explicit expression for the voltage $V_N^D$. Therefore, we define a continuous extension of the voltages $V_j^D,\ j=0,1,\ldots,N$ to a function $V_N^D(t): [0,N] \to \mathbb{R}_+$. That is, for all $j\in \{0,\ldots,N\}$, we show $V_{j}^D = V_N^D(j)$. A suitable and natural extension, cf. \eqref{eq:alternative_second_sum}, is given by
\begin{align}
V_N^D(t) & := 1+k_N \int_0^{\floor{t}}\int_0^{\ceil{s}}\frac{1}{V_N^D({u})} du\ ds,\ t\in[0,N];\label{eq:continuous_extension_VD}
\end{align} where $k_N = r\lambda_N^D$.\ %\frac{a}{N^2}$ with $a\in\mathbb{R}_+$.
Here, the arrival rate $\lambda_N^D$ is the same for all charging stations. See Remark \eqref{remark:equal_arrival_rates}. We include $N$ in the notation of the function $V_N^D(t)$ and $k_N$ to stress its dependence on the number of charging stations $N$.

Furthermore, we define a certain scaling for the continuous extension $V_N^D(t)$ such that it differs from the integral equation
\begin{align}
V(t) = 1+a\int_0^t\int_0^x \frac{1}{V(y)}dy\ dx\label{eq:integral_equation_V}
\end{align} by a negligible term. Here, the value $a$ is chosen such that $k_N = \frac{a}{N^2}$.
\begin{definition}\label{def:scale}
Denote the voltages under the Distflow model in continuous time by $V_N^D(t)$ as in \eqref{eq:continuous_extension_VD}, and scale them by $N$,
\begin{align}
\overline{V}_N^D(t) := V_N^D(Nt), \ t\in[0,1].\label{eq:scaled_Vn}
\end{align}
\end{definition}
This scaling allows us to show, see Section \ref{subsubsec:derivation_of_scaled_version}, that the representation of the scaled version of the voltages $\overline{V}_N^D(t)$ can be written as
\begin{align}
%\overline{V}_N^D(t) & = 1+a\int_0^t\int_0^x \frac{1}{\overline{V}_N^D(\floor{y})}dy\ dx + \overline{V}_N^D(t)-\overline{V}_n^D(t+\frac{1}{n})+ R_N(t)\\
\overline{V}_N^D(t) %& = 1+a\int_0^t\int_0^x \frac{1}{\overline{V}_N^D(\floor{y})}dy\ dx + \overline{R}_N(t)\\
& = 1+a\int_0^{t}\int_0^{s}\frac{1}{\overline{V}_N^D(u)}du\ ds +\overline{R}_N(t)\label{eq:mappingH},
\end{align} where the remainder term $\overline{R}_N(t)$ vanishes as $N\to\infty$. Hence, \eqref{eq:mappingH} can be expected to converge to the integral equation given by \eqref{eq:integral_equation_V}.
%\begin{align}
%V(t) = 1+a\int_0^t\int_0^x \frac{1}{V(y)}dy\ dx.
%V(t) = (HV)(t).\label{eq:integral_equation_1}
%\end{align}
This convergence is exactly shown in Proposition \ref{THM:CONVERGENCE}.

\begin{proposition}\label{THM:CONVERGENCE}
Let $\overline{V}_N^D(t)$, for $t\in[0,1]$, as defined in \eqref{eq:mappingH} and $V(t)$ as in \eqref{eq:integral_equation_V}. %If
%\begin{itemize}
%\item $\Sup{t\in[0,T]} \overline{R}_N(t)\to 0$ as $N\to\infty$, and
%\item the mapping $H:\mathbf{D}_{\geq 1}[0,h]\to\mathbf{D}_{\geq 1}[0,h]$, for $0\leq h<\sqrt{\frac{2}{a}}$, is a contractive mapping.
%\end{itemize}
Then,
\begin{align*}
\sup_{0\leq t\leq 1} |\overline{V}_N^D(t)-V(t)|\to 0\ \text{as}\ N\to\infty.
\end{align*}
\end{proposition} Thus, the appropriately scaled voltages under the Distflow model $\overline{V}_N^D(t)$ can be approximated by $V(t)$ as the number of charging stations goes to infinity. The solution to the integral equation $V(t)$ can be found in Appendix \ref{sec:integral_equation}. Using this approximation, we are able to compute the arrival rate $\lambda_N^D$ such that the maximal allowed voltage drop is attained.

\subsection{Computation of $\lambda_N^D$ such that maximal voltage drop is attained}\label{subsec:computation_lambda}
In this section, we compute the arrival rate $\lambda_N^D$ such that the maximal allowed voltage drop is attained. First, we show that we can approximate $V_N^D$ by $V(1)$ (cf. \eqref{eq:integral_equation_V}). Then, using this approximation, we compute the value $a$, which is related to the arrival rate $\lambda_N^D$ (cf. \eqref{eq:k_N}), that solves the equation $V(1)=\frac{1}{1-\Delta}$. In other words, we compute the value $a$ such that the maximal voltage drop is attained. However, this solution $a$ is computed using the approximation $V(1)$. So, we conclude by showing that this solution $a$ also maximizes the voltage drop using $V_N^D$.

The recursion for the voltages under the Distflow model (cf. Table \ref{tab:load_flow_models}) with arrival rate $\lambda_N^D$, is given by,
\begin{align}
V_0^D & = 1\ \text{and}\ V_1^D = 1 + k_N, \label{eq:DF1_scaled_new}\\
V_{j+1}^D & = 2V_j^D-V_{j-1}^D+\frac{k_N}{V_j^D}, \quad j = 1,\ldots,N-1. \label{eq:DF2_scaled_new}
\end{align} where
\begin{align}
k_N=r\lambda_N^D = \frac{a}{N^2}\geq 0.\label{eq:k_N}
\end{align}
Then, according to the voltage drop constraint in  \eqref{eq:set_of_constraints}, the maximum feasible arrival $\lambda_N^D$ is such that the voltage $V_N^D$, is equal to $1/(1-\Delta)$. In Section \ref{subsec:approximation_continuous counterpart}, we found an equivalent expression for the voltage $V_N^D$, namely $\overline{V}_N^D(1)$. Furthermore, using Proposition \ref{THM:CONVERGENCE}, we found an approximation of the voltage $V_N^D$, namely $V(1)$, as in \eqref{eq:integral_equation_V}.

Differentiating \eqref{eq:integral_equation_V} twice gives the differential equation form of the integral equation \eqref{eq:integral_equation_V}; i.e.,
\begin{align}
V''(t) = \frac{a}{V(t)},\label{eq:diff_eq_Vt1}
\end{align} for $a\in\mathbb{R}_+$, with initial conditions
\begin{align}
V(0)=1\  \text{and}\  V'(0)=0.\label{eq:diff_eq_Vt}
\end{align}
%\begin{align}
%V(1) = f_0(\sqrt{a}) = f_0(N\sqrt{k_N}).
%\end{align}
By Lemma \ref{lemma:solution_f}, the solution to the differential equation \eqref{eq:diff_eq_Vt1}--\eqref{eq:diff_eq_Vt} is given in terms of the function $f_0(x)=\exp(U^2(x))$ where $U(x)$ is related to the inverse of the imaginary error function. From Lemma \ref{lemma:solution_f}, it follows that
\begin{align*}
V(t) = f_0(t\sqrt{a}).
\end{align*} Therefore, an approximation of the voltage $V_N^D$, namely $V(1)$, is given by,
\begin{align}
V(1)=f_0(\sqrt{a}).\label{eq:V1_approx}
\end{align}
Although it is possible to compute the arrival rate $\lambda_N^D$ such that $V_N^D = \frac{1}{1-\Delta}$ using an iterative approach, we can obtain a convenient approximation by using \eqref{eq:V1_approx}. In Section \ref{sec:improvement}, we show how to compute the arrival rate $\lambda_N^D$ using Newton's method and that the approximation \eqref{eq:V1_approx} yields already a good approximation for the arrival rate $\lambda_N^D$.

With the approximation $V(1)$ (as $N\to\infty$) at hand, we compute that value of $a$ that solves the equation $V(1)=\frac{1}{1-\Delta}$, i.e., $f_0(\sqrt{a}) = \frac{1}{1-\Delta}$ and use that $r\lambda_N^D = \frac{a}{N^2}$ (cf. \eqref{eq:k_N}) to relate it to a critical arrival rate. Solving for $a$ such that $f_0(\sqrt{a}) = \frac{1}{1-\Delta}$, yields,
\begin{align*}
a = \left(f_0^{-1}\left(\frac{1}{1-\Delta}\right) \right)^2.
\end{align*} Notice that the inverse function $f_0^{-1}$ is given by
\begin{align*}
f_0^{-1}(x) = \sqrt{\frac{\pi}{2}}\text{erfi}\left(\sqrt{\ln(x)} \right).
\end{align*} Now, for any fixed $a$, we have by Theorem \ref{THM:CONVERGENCE} that $|\overline{V}_N^D(1)-V(1)|\to 0$ uniformly as $N\to\infty$ and by construction that $V(1)=\frac{1}{1-\Delta}$. Now, let \begin{align}
\phi_N := \max\Big\{\overline{V}_N^D(1)\ \text{s.t.}\  \overline{V}_N^D(1)\leq \frac{1}{1-\Delta}\Big\}, \label{eq:phi_N}
\end{align} and
\begin{align*}
\phi := \max\Big\{V(1)\ \text{s.t.}\  V(1)\leq \frac{1}{1-\Delta}\Big\},
\end{align*} denote the maximal voltages and the approximation of the maximal voltages under the voltage drop constraint, respectively. Then, by \cite[Lemma 3.1]{Marti1975}, $\phi_N$ converges to $\phi$ as $N\to\infty$ and since $a=N^2r\lambda_N^D$ maximizes $V(1)$, $a$ maximizes $\overline{V}_N^D(1)$ as well and the limit point $r\lambda_c^D$ of $N^2r\lambda_N^D$ is a solution to \eqref{eq:phi_N}, i.e.\ $N^2r\lambda_N^D \to r\lambda_c^D = \frac{\pi}{2}\text{erfi}^2\left(\sqrt{\ln\left(\frac{1}{1-\Delta} \right)}\right)$, as $N\to\infty$. %Notice, by the triangle inequality, that

\section{Newton's iterations for the Distflow model}\label{sec:improvement}
In Sections \ref{subsec:approximation_continuous counterpart} and \ref{subsec:computation_lambda}, we made an effort to approximate the voltage $V_N^D$ to approximate the arrival rate $\lambda_N^D$ such that $V_N^D=\frac{1}{1-\Delta}$ easier. However, it is possible to compute the arrival rate $\lambda_N^D$ such that $V_N^D=\frac{1}{1-\Delta}$ using the recursive equations \eqref{eq:DF1_scaled_new} and \eqref{eq:DF2_scaled_new} using an iterative approach. %However, as $N\to\infty$, we obtained a limit theorem for the voltages under the Distflow model in Theorem \ref{THM:DISTFLOW}. %Therefore, we first compute the maximal arrival rate under the approximation given by Theorem \ref{THM:DISTFLOW} as sketched at the end of Section \ref{section:proof_distflow}. %For small $\Delta$, we have $\Delta \approx \frac{a}{2}$.
%According to Theorem \ref{THM:CONVERGENCE} we have, as $N\to\infty$, that $V_n^D\to V(1)=f_0(\sqrt{a})$. We expect $f_0(\sqrt{a})$ to be a good approximation of $V_n^D$ as $N\to\infty$. The approximation $f_0(\sqrt{a})$ allows us to explicitly compute the value of $a$ such that $f_0(\sqrt{a}) = \frac{1}{1-\Delta}$.
%We show that we can use the recursion values $V_n^D, n=0,\ldots,N$ to compute the maximal feasible arrival rate under the Distflow model.
In other words, we show how to compute, using the recursive equations \eqref{eq:DF1_scaled_new} and \eqref{eq:DF2_scaled_new}, for a fixed value of the number of charging stations $N$ and a maximum allowed voltage drop $\Delta>0$, the number $a$ such that the voltage drop between the root node and node $N$ is exactly equal to $\Delta$. Furthermore, we present some numerical tests showing the convergence to the number $a$.

In what follows, we propose to use Newton's method to find the solution $a$ that yields $V_N^D=\frac{1}{1-\Delta}$. %We assume here as given that $V_n^D$ increases in $a\in\mathbb{R}_+$ for $n=1,\ldots,N$ \red{Non-trivial issue, can show it when $a<2$}.
%First, we rewrite the sequence $V_n^D$ for $n=0,1,\ldots,N$ defined in \eqref{eq:DF1_scaled_new} and \eqref{eq:DF2_scaled_new} in terms of $a$. We get
%\begin{align}
%V_0^D=1,V_1^D = 1+\frac{a}{N^2}\label{eq:Vn_initial_a}
%\end{align} and
%\begin{align}
%V_{n+1}^D=2V_n^D-V_{n-1}^D+\frac{a}{N^2V_n^D},\quad n=1,2\ldots,N-1,\label{eq:Vn_sequence_a}
%\end{align} for fixed $a\in\mathbb{R}_+$.
We need to initialize a particular $a_0$ to apply Newton's method.  Here, we are led by Theorem \ref{THM:CONVERGENCE}. We look for an initial guess $a_0$ such that $V(1) = f_0\left(\sqrt{a}\right) = \frac{1}{1-\Delta}$, we choose
\begin{align*}
a_0 & = \left(f_0^{-1}\left(\frac{1}{1-\Delta} \right)\right)^2 \nonumber\\
& = \frac{\pi}{2}\text{erfi}^2\left(\sqrt{\ln\left(\frac{1}{1-\Delta}\right)} \right),
\end{align*} and iterate according to
\begin{align}
a_{j+1}=a_j-\frac{V_N^D-(\frac{1}{1-\Delta})}{Y_N^D/N^2},\quad j=0,1,\ldots\label{eq:iteration_step}
\end{align} where $Y_N^D$ denotes the derivative of $V_N^D$ with respect to $k_N$. Note that $Y_N^D/N^2$ denotes the derivative of $V_N^D$ with respect to $a$. We need to compute $Y_N^D$, and this can be done by differentiating in \eqref{eq:DF1_scaled_new} and \eqref{eq:DF2_scaled_new} with respect to $k_N$, and so that
\begin{align}
Y_0^D=0, Y_1^D=1\label{eq:Wn_initial_a}
\end{align} and
\begin{align}
Y_{n+1}^D - 2Y_n^D+Y_{n-1}^D = \frac{1}{V_n^D}-\frac{k_NY_n^D}{(V_n^D)^2}.\label{eq:diff_W_newton}
\end{align} Observe that the right-hand side of \eqref{eq:diff_W_newton} also involves $V_n^D$. Thus, one should compute $V_n^D$ and $Y_n^D$ simultaneously and recursively according to the initialization in \eqref{eq:DF1_scaled_new} and \eqref{eq:Wn_initial_a}, and the recursion step for $n=1,\ldots,N-1$ in \eqref{eq:DF2_scaled_new} and \eqref{eq:diff_W_newton}. Furthermore, the iteration step in \eqref{eq:iteration_step} is well defined if $Y_N^D$ is positive. This will, for reasonable values of $a$, be shown in Lemma \ref{lemma:Wn_positive}. Before we move to Lemma \ref{lemma:Wn_positive}, we comment on the range of reasonable values for $a$. Notice, from the condition that $f_0(\sqrt{a}) = \frac{1}{1-\Delta}$, it follows from a Taylor expansion about $\Delta=0$ that
\begin{align*}
\int_0^{\sqrt{\ln(\frac{1}{1-\Delta})}}\exp(u^2)du = \sqrt{\frac{a}{2}}.
\end{align*} For small $\Delta$, which is usually the case (cf. Section \ref{subsec:distribution}), it follows that $\Delta\approx \frac{a}{2}$, which can also be seen in Tables \ref{tab:a_first}, \ref{tab:a_0.05} and \ref{tab:a_0.1}. Therefore, as we consider $0<a<2$, these are reasonable values given our application.

\begin{lemma}\label{lemma:Wn_positive}
Consider the recursion in \eqref{eq:Wn_initial_a} and \eqref{eq:diff_W_newton} with $0<a<2$. Then the sequence $Y_0,Y_1,\ldots,Y_N$ is positive, increasing and convex.
\end{lemma}

\begin{proof}
The right-hand side of \eqref{eq:diff_W_newton} is positive for all $n=1,2,\ldots,N-1$ such that $Y_n^D < V_n^D/k_N$. Therefore, since, $Y_0^D = 0 < 1 = Y_1$, the sequence $Y_n^D, n=1,2,\ldots$ is increasing and convex (and thus positive) as long as $Y_n^D < V_n^D/k_N$. For such $n$, we have certainly
\begin{align*}
Y_{n+1}^D-2Y_n^D+Y_{n-1}^D = \frac{1}{V_n^D}\left(1-\frac{k_NY_n^D}{V_n^D} \right)\leq 1,
\end{align*} since
\begin{align*}
\frac{k_NY_n^D}{(V_n^D)^2}<\frac{1}{V_n^D}\leq 1.
\end{align*} Here we used Lemma \ref{lemma:V_increasing}. So, using $Y_0^D=0,Y_1^D=1$ and induction, we get
\begin{align*}
Y_{n+1}^D-Y_n^D \leq (n+1),
\end{align*} and consequently for these $n$, by summing from $m=0,\ldots,n$,
\begin{align*}
Y_{n+1}^D = \sum_{m=0}^n Y_{m+1}^D-Y_m^D \leq \sum_{m=0}^n m+1 = \frac{1}{2}(n+1)(n+2).
\end{align*} Therefore, the condition $Y_n^D<V_n^D/k_N$ continues to be satisfied certainly when
\begin{align}
\frac{1}{2}n(n+1)<V_n^D/k_N.\label{eq:certainly_satisfied}
\end{align} The right-hand side of \eqref{eq:certainly_satisfied} is at least equal to $N^2/a$, and the left-hand side of \eqref{eq:certainly_satisfied} is at most equal to $\frac{1}{2}(N-1)N$, since $n=0,\ldots,N-1$. %1/2n(n+1)<1/2N(N-1)<N^2/a<V_n/k_N
Thus \eqref{eq:certainly_satisfied} holds for all $n=1,2,\ldots,N-1$ when $\frac{1}{2}(N-1)< N/a$, i.e., when
\begin{align*}
a < \frac{2N}{N-1}.
\end{align*}
\end{proof}

\subsection{Numerical validation of Newton's method}

Below we summarize some numerical tests. For fixed values of $a$, ranging from 0.01 to 0.1, and several values of $N=10,\ldots,10^5$, we compute $V_N^D$, and we put
\begin{align*}
\frac{1}{1-\Delta} = V_N^D,
\end{align*} and using the Newton scheme above we compute the solution $\bar{a}$ such that $V_N^D = \frac{1}{1-\Delta}$ with stopping criteria $\frac{|a_{j+1}-a_j|}{a_j}<10^{-10}$.

\begin{table}[h!]
\begin{center}
 \begin{tabular}{|c r r r r|}
 \hline
 \multicolumn{5}{|c|}{$a=0.01$} \\
 \hline
 $N$ & $\frac{1}{1-\Delta}$ & $a_0$ & $\bar{a}$ & $\#$Newton iterations  \\ [0.5ex]
 \hline\hline
 10 & 1.005495062463669 & 0.011000182805825 & 0.009999999999999 & 3\\
 \hline
 $10^2$ & 1.005045760405502 & 0.010100001678824 & 0.009999999999995 & 3\\
 \hline
 $10^3$ & 1.005000834727210 & 0.010010000022962
 & 0.010000000000980
 & 3\\
 \hline
 $10^4$ & 1.004996342221457 & 0.010001000055592 & 0.009999999999841 & 8\\
 \hline
 $10^5$ & 1.004995909696177 & 0.010000133565834 & 0.009999999993546 & 9\\
 \hline
\end{tabular}
\end{center}

\caption{Illustration of Newton scheme in the computation of the solution $\bar{a}$ such that $V_N^D=\frac{1}{1-\Delta}$, for $a=0.01$.}
\label{tab:a_first}
\end{table}

\begin{table}[h!]
\begin{center}
 \begin{tabular}{|c r r r r|}
 \hline
 \multicolumn{5}{|c|}{$a=0.05$} \\
 \hline
 $N$ & $\frac{1}{1-\Delta}$ & $a_0$ & $\bar{a}$ & $\#$Newton iterations  \\ [0.5ex]
 \hline\hline
 10 & 1.027377786724925 & 0.055004518819866& 0.049999999999998 & 3\\
 \hline
 $10^2$ & 1.025144992180518 & 0.050500041530882 & 0.050000000000026
 & 3\\
 \hline
 $10^3$ & 1.024921824633206 & 0.050050000413244

 & 0.049999999999354 & 2\\
 \hline
 $10^4$ & 1.024899508844976 & 0.050005000058683
& 0.049999999998112 & 2\\
 \hline
 $10^5$ & 1.024897282801763 & 0.050000511203957
 & 0.050000000039374 & 4\\
 \hline
\end{tabular}
\end{center}
\caption{Illustration of Newton scheme in the computation of the solution $\bar{a}$ such that $V_N^D=\frac{1}{1-\Delta}$, for $a=0.05$.}
\label{tab:a_0.05}
\end{table}

\begin{table}[h!]
\begin{center}
 \begin{tabular}{|c r r r r|}
 \hline
 \multicolumn{5}{|c|}{$a=0.1$} \\
 \hline
 $N$ & $\frac{1}{1-\Delta}$ & $a_0$ & $\bar{a}$ & $\#$Newton iterations  \\ [0.5ex]
 \hline\hline
 10 & 1.054517088899833 & 0.110017830743300& 0.099999999999999 & 3\\
 \hline
 $10^2$ & 1.050084740193820 & 0.101000164022137 & 0.100000000000124 & 3\\
 \hline
 $10^3$ & 1.049641947170216 & 0.100100001626823 & 0.099999999999965 & 3\\
 \hline
 $10^4$ & 1.049597671662610 & 0.100010000152368 & 0.100000000003042 & 4\\
 \hline
 $10^5$ & 1.049593246696348 & 0.100001005329048 & 0.099999999991939 & 4\\
 \hline
\end{tabular}
\end{center}
\caption{Illustration of Newton scheme in the computation of the solution $\bar{a}$ such that $V_N^D=\frac{1}{1-\Delta}$, for $a=0.1$.}
\label{tab:a_0.1}
\end{table}

It appears that the initialization $a_0$ is already a good approximation for $a$ as the differences between $a_0$ and $a$ for all situations are small. For example, the relative error of the initialization $a_0$ with respect to the true value $a=0.01$ is $10\%$ in the case of $N=10$ and decreases even further when $N$ increases. Moreover, the number of iterations needed to converge to a final estimate $\bar{a}$ is small.

\begin{appendices}
\section{Equivalence for the Distflow model}\label{sec:equivalence}
In Section \ref{sec:model_description}, we find the equivalence \eqref{eq:equivalence} between the voltages under the Distflow model at the root node and the node at the end of the line. To prove the equivalence in Proposition \ref{prop:equivalence}, we need Lemmas \ref{lemma:alternative_Vn} and \ref{lemma:v_increasing}.

\subsection{Proof of Lemma \ref{lemma:alternative_Vn}}
\begin{proof} Both identities follow from the initialization in \eqref{eq:DF1} and the definition of the sequence in \eqref{eq:DF2}.
\begin{enumerate}
\item We have from \eqref{eq:DF2} for $j=1,2,\ldots,N-1$,
\begin{align}
{V_{j+1}}^D-{V_j}^D = {V_j}^D-{V_{j-1}}^D+\frac{rp_j}{{V_j}^D}.\label{eq:Vn_summing}
\end{align} We have from the identity in \eqref{eq:Vn_summing} by summation that
\begin{align}
{V_{i+1}}^D-{V_i}^D & = ({V_1}^D-{V_0}^D) + \sum_{j=1}^i \frac{rp_j}{{V_j}^D}\nonumber\\
& = \frac{rp_0}{{V_0}^D} + \sum_{j=1}^i \frac{rp_j}{{V_j}^D}\nonumber\\
& = \sum_{j=0}^i \frac{rp_j}{{V_j}^D},\quad i = 0,\ldots,N-1.\label{eq:alternative_Vn_sequence}
\end{align}
\item We have from the identity in \eqref{eq:alternative_Vn_sequence} by summation that
\begin{align*}
{V_{n}}^D-{V_0}^D = \sum_{i=0}^{n-1}\sum_{j=0}^i \frac{rp_j}{{V_j}^D}, \quad n = 0,\ldots,N.
\end{align*} Hence, with the initialization in \eqref{eq:DF1} statement (2) follows immediately.
\end{enumerate}
\end{proof}

\subsection{Proof of Lemma \ref{lemma:v_increasing}}
\begin{proof}
This is true for $j=0$, since $\frac{dv_j}{d{V_0}^D} = \frac{d{V_0}^D}{d{V_0}^D}=1.$ Suppose that there is some $0\leq J\leq n-1$ such that $0\leq \frac{dv_j}{d{V_0}^D}\leq 1$ for all $0\leq j\leq J$, then by Lemma \ref{lemma:alternative_Vn} and the definition of $v_j$,
\begin{align}
\frac{dv_{J+1}}{d{V_0}^D} & = \frac{d}{d{V_0}^D}\left({V_0}^D + \sum_{i=0}^J\sum_{j=0}^i \frac{rp_j}{v_j({V_0}^D)} \right) \nonumber\\
& = 1-\sum_{i=0}^J\sum_{j=0}^i \frac{rp_j}{v_j^2({V_0}^D)}\frac{dv_j}{d{V_0}^D}\nonumber\\
& = 1-\sum_{i=0}^J\sum_{j=0}^i \frac{rp_j}{{({V_j}^D})^2}\frac{dv_j}{d{V_0}^D} \label{eq:derivative_v_j+1}.
\end{align} First, we show that $\frac{dv_{J+1}}{d{V_0}^D}\geq 0$. Recall that we establish in Lemma \ref{lemma:V_increasing} that the sequence of voltages ${V_j}^D, j=0,\ldots,N$ is an increasing sequence. This implies that $-\frac{1}{V_0^D}\leq -\frac{1}{V_j^D}$ for all $j=0,\ldots,N$. Using the fact that $-\frac{1}{V_0^D}\leq -\frac{1}{V_j^D}$ for all $j=0,\ldots,N$ and $\frac{dv_j}{d{V_0}^D}\leq 1$ yields,
\begin{align*}
\frac{dv_{J+1}}{d{V_0}^D} & \geq 1-\frac{1}{{V_0}^D}\sum_{i=0}^J\sum_{j=0}^i\frac{rp_j}{{V_j}^D}.
\end{align*} By \eqref{eq:alternative_second_sum} in Lemma \ref{lemma:alternative_Vn}, we have consequently
\begin{align*}
\frac{dv_{J+1}}{d{V_0}^D} & \geq 1-\frac{1}{{V_0}^D}\left({V_{J+1}}^D-{V_0}^D\right)\\
& = 2-\frac{{V_{J+1}}^D}{{V_0}^D}.
\end{align*} Combining the fact that the ${V_j}^D$ is an increasing sequence and the assumption that ${V_N}^D\leq 2{V_0}^D$, gives the inequality $-{V_{J+1}}^D \geq -2{V_0}^D$. Hence,
\begin{align*}
\frac{dv_{J+1}}{d{V_0}^D} & \geq 0.
\end{align*} Now, we show that $\frac{dv_{J+1}}{d{V_0}^D}\leq 1$. By the induction hypothesis, we have $-\frac{dv_j}{d{V_0}^D}\leq 0$ for all $j=0,\ldots,N$. Starting from Equation \ref{eq:derivative_v_j+1}, we immediately get,
\begin{align*}
\frac{dv_{J+1}}{d{V_0}^D} & = 1-\sum_{i=0}^J\sum_{j=0}^i \frac{rp_j}{{({V_j}^D})^2}\frac{dv_j}{d{V_0}^D} \leq 1.
\end{align*}
\end{proof}

\subsection{Proof of Proposition \ref{prop:equivalence}}
\begin{proof}
In order to prove the implication from left to right, we take the negation of the right-hand side of \eqref{eq:equivalence} and show that this implies the negation of the left-hand side of \eqref{eq:equivalence}. Suppose $v_N({V'_0}^D)>c$. By Lemma \ref{lemma:v_increasing}, $v_N$ is an increasing function in ${V'_0}^D$, so $c<v_N((1-\Delta)c)\leq v_N(x)$ for $x\geq (1-\Delta)c$. Hence, there exists no $x\geq (1-\Delta)c$ such that $v_N(x)=c$.

For the other implication, we first observe that $v_N$ is a continuous function, because it is a composition of continuous functions itself. To prove the implication from the right-hand side of \eqref{eq:equivalence} to the left-hand side of \eqref{eq:equivalence}, we assume that at ${V'_0}^D=(1-\Delta)c$, we have $v_N({V'_0}^D)\leq c$. By Lemma \ref{lemma:v_increasing}, we know $v_N((1-\Delta)c)\leq c$. Then, by the intermediate value theorem, we know that there exists $(1-\Delta)c\leq x\leq c$ such that $v_N(x)=c$.
\end{proof}

\section{Stability result}\label{sec:stability_result}
\subsection{Finite network: model}\label{subsection:model}
Assume that we have an arbitrary network of queues such that individual instantaneous service rates may depend on the state of the entire system. The network state is represented as a set $\mathbf{X}(t)=(X_i(t),i\in N)$ of the queue lengths $X_i$ at the network nodes $i\in N$ at time $t$. Arrivals into queue $i$ occur according to a Poisson process with a constant rate $\lambda_i$, %smaller than 1 is not necessary, see email seva
independent of all the other processes. The instantaneous service rate $\mu_i$ of a node represents the intensity of the Poisson process modeling departures (service completions) of the node. In this case, the service rate allocation algorithm $\boldsymbol{\psi}(\mathbf{X}(t))$, mapping a network state $\mathbf{X}(t)$ into a set of instantaneous service rates $\boldsymbol{\mu}$, is all that is needed to specify the service allocation algorithm. %1he instantaneous departure rate from queue $i$ at time $t$, conditioned on the state $\mathbf{X}(t)$, is $\psi_i(\mathbf{X}(t))$; more precisely, the number of departures up to time $t$ is $\Pi_i(\int_0^t \psi_i(\mathbf{X}(\tau))d\tau)$, where $\Pi_i(\cdot)$ are independent unit-rate Poisson processes.

For simplicity, in this section, we drop the dependence on time $t$ from the notation and we assume that
\begin{align}
\boldsymbol{\psi}(\boldsymbol{x})\in\argmax_{\boldsymbol{\mu}\in\mathcal{C}}\sum_{i}g(x_i)h(\mu_i),\label{eq:argmax}
\end{align} where the set $\mathcal{C}$ is compact. We impose, in addition, the following conditions:
\begin{itemize}
\item Condition ($H$): the function $h:[0,\infty)\to\mathbb{R}$ is strictly increasing, differentiable and concave (both the cases $\lim_{y\downarrow 0}h(y)=h(0)>-\infty$ and $\lim_{y\downarrow 0}h(y)=-\infty$ are allowed); and
\item Condition ($G$): the function $g:\mathbb{Z}_{+}\to[0,\infty)$ is strictly increasing and such that
\begin{align}
\frac{g(y)}{\Delta(y)}\to\infty\ \text{as}\ y\to\infty,\label{eq:g1}
\end{align} where $\Delta(y)=g(y+1)-g(y)$. Note that \eqref{eq:g1} is equivalent to
\begin{align*}
\frac{g(y+1)}{g(y)}\to 1,\ \text{as}\ y\to\infty.
\end{align*}
\end{itemize} We are going to assume that
\begin{align}
\text{there exists}\ \boldsymbol{\nu}\in\mathcal{C}\ \text{such that}\ \boldsymbol{\lambda}<\boldsymbol{\nu}.\label{existence_lambda}
\end{align} Note that stability condition \eqref{existence_lambda} is equivalent to
\begin{align}
\text{there exists}\ \boldsymbol{\nu}\in\mathcal{C}_1\ \text{such that}\ \boldsymbol{\lambda}<\boldsymbol{\nu},\label{existence_lambda_C1}
\end{align} where
\begin{align*}
\mathcal{C}_1 = \{\boldsymbol{\nu}: \exists \boldsymbol{\mu}\in\mathcal{C}\ \text{s.t.}\  \boldsymbol{\nu}\leq \boldsymbol{\mu}\}.
\end{align*} Indeed, for the implication from \eqref{existence_lambda_C1} to \eqref{existence_lambda}, if there exists $\boldsymbol{\nu}\in\mathcal{C}_1$ such that $\boldsymbol{\lambda}<\boldsymbol{\nu}$, then $\boldsymbol{\lambda}<\boldsymbol{\nu}\leq \boldsymbol{\mu}$ for $\boldsymbol{\mu}\in\mathcal{C}$, so there exists $\boldsymbol{\mu}\in\mathcal{C}$ such that $\boldsymbol{\lambda}<\boldsymbol{\mu}$. For the implication from \eqref{existence_lambda} to \eqref{existence_lambda_C1}, assume there exists $\boldsymbol{\nu}\in\mathcal{C}$ such that $\boldsymbol{\lambda}<\boldsymbol{\nu}=\boldsymbol{\nu}$, hence $\boldsymbol{\nu}\in\mathcal{C}_1$ such that $\boldsymbol{\lambda}<\boldsymbol{\nu}$.

Note that $\mathcal{C}_1$ is coordinate-convex by construction. Note also that if \begin{align*}
\sum_i g(x_i)h(\psi_i(\boldsymbol{x}))\geq \sum_i g(x_i)h(\mu_i)
\end{align*} for all $\boldsymbol{\mu}\in\mathcal{C}$, then
\begin{align*}
\sum_i g(x_i)h(\psi_i(\boldsymbol{x}))\geq \sum_i g(x_i)h(\nu_i)
\end{align*} for all $\boldsymbol{\nu}\in\mathcal{C}_1$, since the function $h$ is strictly increasing and $g$ is non-negative. Hence,
\begin{align*}
\boldsymbol{\psi}(\boldsymbol{x})\in\argmax_{\boldsymbol{\nu}\in\mathcal{C}_1}\sum_i g(x_i)h(\nu_i).
\end{align*} Then, the rest of the proof of \cite[Theorem 11]{Shneer2018} follows.%We will also denote

\section{Convergence of $\overline{V}_N^D(t)$ to $V(t)$ as $N\to\infty$}\label{sec:conv_VN_to_V(1)}
An important step in the proof of Theorem \ref{THM:DISTFLOW} is the approximation of the voltage $V_N^D$ by a continuous counterpart. In this section, we prove the convergence of a scaled version of the voltage $V_N^D$, i.e.\ $\overline{V}_N^D(t)$, as introduced in Equation \eqref{eq:scaled_Vn}, to the continuous function $V(t)$, as introduced in Equation \eqref{eq:integral_equation_V}, as the network size $N$ grows to infinity (cf. Proposition \ref{THM:CONVERGENCE}) and present its numerical validation.

\subsection{Proof of Proposition \ref{THM:CONVERGENCE}}
The proof consists of several steps that we discussed globally in Section \ref{subsec:approximation_continuous counterpart} already, and for which we now present the details.

In Section \ref{subsubsec:derivation_of_scaled_version}, we rewrite the scaled version of the voltage $\overline{V}_N^D(t)$, in the same form as the continuous function $V(t)$. We show that the continuous extension $V_N^D(t)$ (as in Equation \eqref{eq:continuous_extension_VD}) of the discrete voltages $V_n^D$ is such that $V_n^D = V_N^D(n)$ for all $n=0,1,\ldots,N$. Then, we scale the continuous extension $V_N^D(t)$ by $N$, such that $\overline{V}_N^D(t) = V_N^D(Nt)$ (as in Equation \eqref{eq:scaled_Vn}) and show that this scaled version $\overline{V}_N^D(t)$ can be written as $\overline{V}_N^D(t) = 1+a\int_0^t\int_0^s \frac{1}{\overline{V}_N^D(u)}du\ ds + \overline{R}_N(t)$ (as in Equation \eqref{eq:mappingH}).

Then, for convenience, we introduce the map $H$ (cf. Definition \ref{def:mappingH}), that allows us to write,
\begin{align}
\overline{V}_N^D(t)-V(t) = \overline{R}_N(t) + (H\overline{V}_N^D)(t)-(HV)(t).\label{eq:ideas_proof_conv}
\end{align} The goal is to show that the supremum over the interval $[0,1]$ of the absolute difference between $\overline{V}_N^D(t)$ and $V(t)$ goes to zero as the size of the network $N$ goes to infinity. Having \eqref{eq:ideas_proof_conv} in mind, we consider the remainder term $\overline{R}_N(t)$ and the term $(H\overline{V}_N^D)(t)-(HV)(t)$ of the right-hand side of \eqref{eq:ideas_proof_conv} separately.

In Section \ref{subsubsec:convergence_of_remainder}, we show that this remainder term $\overline{R}_N(t)$ vanishes as $N\to\infty$ for every $t$ in the interval $[0,1]$ and in Section \ref{subsubsec:contraction}, we show that the map $H$ is a contractive mapping on the interval $[0,h]$, where $h<1$. This means,
\begin{align*}
\sup_{0\leq t\leq h} |(H\overline{V}_n^D)(t)-(HV)(t)|\leq \kappa(h)\sup_{0\leq t\leq h} |V(t)-\overline{V}_N^D(t)|
\end{align*} for $\kappa(h)<1$. Combining the results in Sections \ref{subsubsec:convergence_of_remainder} and \ref{subsubsec:contraction}, we then find that
\begin{align*}
\sup_{0\leq t\leq h} |\overline{V}_N^D(t)-V(t)| \to 0\ \text{as}\ N\to\infty.
\end{align*}

In Section \ref{subsubsec:conv_on_diff_intervals}, we show by induction that the convergence of the supremum of the absolute difference between $\overline{V}_N^D(t)$ and $V(t)$ on the interval $[0,h]$ can be extended to the interval $[0,1]$, as desired in Proposition \ref{THM:CONVERGENCE}.

\subsubsection{Derivation of $\overline{V}_N^D(t)$ as in Equation \eqref{eq:mappingH}}\label{subsubsec:derivation_of_scaled_version}
In this section, we first derive an expression for the voltages under the Distflow model as a double summation, instead of the voltages given by Equations \eqref{eq:DF1_scaled_new} and \eqref{eq:DF2_scaled_new}. Then, we show, in Lemma \ref{lemma:integral_representation}, that the function $V_N^D(t)$ (as in Equation \eqref{eq:scaled_Vn}) is a continuous extension of the discrete voltages $V_n^D$, using the alternative expression for the voltages from Lemma \ref{lemma:alternative_Vn}. %i.e.,
%\begin{align}
%V_N^D(t) =
%\begin{cases}
%1, & \text{for}\ 0\leq t<1,\\
%V_1^D, & \text{for}\ 1\leq t<2,\\
%\ldots,\\
%V_{N}^D & \text{for}\ N\leq t<N+1,
%\end{cases}
%\end{align} is such that $V_n^D=V_N^D(n)$ for all $n=0,\ldots,N$.

For convenience, we state the alternative expression for the voltages under the Distflow model from Lemma \ref{lemma:alternative_Vn}; i.e., for $n=0,\ldots,N$,
\begin{align}
V_{n}^D = 1+k_N\sum_{i=0}^{n-1}\sum_{j=0}^i \frac{1}{V_j^D},\label{eq:alternative_Vn_seq}
\end{align} with the convention $\sum_{i=0}^{-1}\cdot = 0$ and $k_N = r\lambda_N^D$. Again, observe that the alternative expression in \eqref{eq:alternative_Vn_seq}, is only defined for integer values. In the following lemma, we define a continuous extension of the alternative expression for the voltages under the Distflow model.
\begin{lemma}[Continuous extension]\label{lemma:integral_representation}
The sequence $V_n^D, n=0,1,\ldots,N$ defined in \eqref{eq:alternative_Vn_seq} can be represented by
\begin{align}
V_N^D(t) & := 1+k_N \int_0^{\floor{t}}\int_0^{\ceil{s}}\frac{1}{V_N^D(u)}du\ ds,\ t\in[0,N];\label{eq:V_n(t+1)}
\end{align} i.e., for all integer values $n=\floor{t}$, $V_N^D(n) = V_n^D$.
\end{lemma}

\begin{proof} Let $t\in[0,N]$. Due to the definition of $V_N^D(t)$ in \eqref{eq:V_n(t+1)} we have
\begin{align}
V_N^D(t) & = 1+k_N \int_0^{\floor{t}}\int_0^{\ceil{s}}\frac{1}{V_N^D(u)}du\ ds
= V_N^D(\floor{t}). \label{eq:floor_is_integer}
\end{align} Thus, for all integer values $n = \floor{t}$, we have the equality
\begin{align*}
V_N^D(n) & = 1+k_N \int_0^{n}\int_0^{\ceil{s}}\frac{1}{V_N^D(u)}du\ ds.
\end{align*} We show $V_N^D(n)=V_n^D$ for all $n=0,\ldots,N$ by strong induction. Note that, for $n=0$, we have
\begin{align*}
V_N^D(0) = 1 = V_0^D
\end{align*} by definition. Now, assume that,
\begin{align}
V_N^D(n)=1+k_N\int_0^n\int_0^{\ceil{s}}\frac{1}{V_N^D(u)}du\ ds = V_n^D,\label{eq:induction_step}
\end{align} for $0\leq n\leq m$. Then, we need to show that $V_N^D(m+1)=V_{m+1}^D$. By definition and linearity of the integral, we have
\begin{align*}
V_N^D(m+1) & = 1+k_N\int_0^{m+1}\int_0^{\ceil{s}}\frac{1}{V_N^D(u)}du\ ds \\
& = 1+k_N\int_0^{m}\int_0^{\ceil{s}}\frac{1}{V_N^D(u)}du\ ds + k_N\int_m^{m+1}\int_0^{\ceil{s}}\frac{1}{V_N^D(u)}du\ ds.
\end{align*} By the induction hypothesis in \eqref{eq:induction_step}, we have
\begin{align}
V_N^D(m+1) = V_m^D + k_N\int_m^{m+1}\int_0^{\ceil{s}}\frac{1}{V_N^D(u)}du\ ds.\label{eq:V_n(t+1)_in_n+1}
\end{align} For $s\in[m,m+1]$, the inner integral in \eqref{eq:V_n(t+1)_in_n+1} simplifies to
\begin{align*}
\int_0^{m+1}\frac{1}{V_N^D(u)}du & = \int_0^1 \frac{1}{V_N^D(u)}du + \int_1^2 \frac{1}{V_N^D(u)}du + \cdots + \int_{m}^{m+1} \frac{1}{V_N^D(u)}du.
\end{align*} Now, again by the induction hypothesis in \eqref{eq:induction_step} and \eqref{eq:floor_is_integer}, we have
\begin{align}
\int_0^{m+1}\frac{1}{V_N^D(u)}du = \frac{1}{V_0^D} +\frac{1}{V_1^D}+\cdots+\frac{1}{V_m^D} = \sum_{j=0}^m \frac{1}{V_j^D}.\label{eq:integral_to_sum}
\end{align} Continuing from \eqref{eq:V_n(t+1)_in_n+1}, and using the result in \eqref{eq:integral_to_sum}, leads to
\begin{align*}
V_N^D(m+1) & = V_m^D + k_N\int_m^{m+1} \sum_{j=0}^m \frac{1}{V_j^D} ds \\
& = V_m^D + k_N\sum_{j=0}^m \frac{1}{V_j^D}.
\end{align*} Here, we recognize Equation \eqref{eq:alternative_first_sum} from Lemma \ref{lemma:alternative_Vn}. Hence,
\begin{align*}
V_N^D(m+1) = V_{m+1}^D.
\end{align*} This concludes the proof. %Then, for $s\in[0,n]$, the double integral in \eqref{eq:V_n(t+1)_in_n+1} becomes
%\begin{align}
%\int_0^{n}&\left(\sum_{j=0}^{\ceil{s}-1}\frac{1}{V^D(j)}\right) ds \\
%& = \int_0^1 \left(\sum_{j=0}^{\ceil{s}-1} \frac{1}{V^D(j)}\right)ds + \int_1^2 \left(\sum_{j=0}^{\ceil{s}-1} \frac{1}{V^D(j)}\right)ds + \ldots + \int_{n-1}^{n} \left(\sum_{j=0}^{\ceil{s}-1} \frac{1}{V^D(j)}\right)ds \\
%& = \int_0^1 \left(\sum_{j=0}^{0} \frac{1}{V^D(j)}\right)ds + \int_1^2 \left(\sum_{j=0}^{1} \frac{1}{V^D(j)}\right)ds + \ldots + \int_{n-1}^{n} \left(\sum_{j=0}^{n} \frac{1}{V^D(j)}\right)ds \\
%& = \sum_{j=0}^{0} \frac{1}{V^D(j)} + \sum_{j=0}^{1} \frac{1}{V^D(j)} + \ldots + \sum_{j=0}^{n-1} \frac{1}{V^D(j)} \\
%& = \sum_{i=0}^{n-1}\sum_{j=0}^i \frac{1}{V^D(j)}.
%\end{align} Hence, for $t\in[0,N]$, we arrive at the equality
%\begin{align}
%V_N^D(t) = 1+k_N\sum_{i=0}^{n-1}\sum_{j=0}^i \frac{1}{V_j^D} = V_{n}^D.
%\end{align}
\end{proof}

%\subsubsection{Convergence of ${V}_N^D(t)$ as $N\to\infty$}

%\begin{theorem}\label{THM:CONVERGENCE}
%Let $t\in[0,T]$ fixed. Let $\overline{V}_N^D(t)$ as defined in \eqref{eq:mappingH} and $V(t)$ as in \eqref{eq:integral_equation_1}. %If
%\begin{itemize}
%\item $\Sup{t\in[0,T]} \overline{R}_N(t)\to 0$ as $N\to\infty$, and
%\item the mapping $H:\mathbf{D}_{\geq 1}[0,h]\to\mathbf{D}_{\geq 1}[0,h]$, for $0\leq h<\sqrt{\frac{2}{a}}$, is a contractive mapping.
%\end{itemize}
%Then,
%\begin{align}
%\sup_{0\leq t\leq T} |\overline{V}_N^D(t)-V(t)|\to 0\ \text{as}\ N\to\infty.
%\end{align}
%\end{theorem}

So far, we showed that the function ${V}_N^D(t)$ is a continuous extension of the discrete voltages $V_n^D$. Now, we scale the continuous extension $V_N^D(t)$, by $\overline{V}_N^D(t) = V_N^D(Nt)$ (cf. Definition \ref{def:scale}), into an integral equation that, as we show later, differs from the integral equation in \eqref{eq:integral_equation_V} by terms that vanish as $N\to\infty$. We have, for $t\in[0,1]$ and $k_N = \frac{a}{N^2}$, that
\begin{align}
\overline{V}_N^D(t+\frac{1}{N}) & =  1+\frac{a}{N^2}\int_0^{\floor{Nt}+1}\int_0^{\ceil{s}}\frac{1}{V_N^D(u)}du\ ds\nonumber\\
& = 1+\frac{a}{N^2}\int_0^{\floor{Nt}+1}\int_0^{\ceil{s}}\frac{1}{\overline{V}_N^D\left(\frac{u}{N}\right)}du\ ds.\label{eq:change_variables_V_n}
\end{align} Substituting $s=Nx$ and $u=Ny$ in the integral in \eqref{eq:change_variables_V_n}, we get
\begin{align}
\overline{V}_N^D(t+\frac{1}{N})&=1+\frac{a}{N^2} N^2\int_{0}^{\frac{\floor{Nt}+1}{N}}\int_0^{\frac{\ceil{Nx}}{N}}\frac{1}{\overline{V}_N^D(y)}dy\ dx\nonumber\\
& = 1+a\int_{0}^{t}\int_0^{\frac{\ceil{Nx}}{N}}\frac{1}{\overline{V}_N^D(y)}dy\ dx + a\int_{t}^{\frac{\floor{Nt}+1}{N}}\int_0^{\frac{\ceil{Nx}}{N}}\frac{1}{\overline{V}_N^D(y)}dy\ dx.\label{eq:change_var_V_n}
\end{align} Finally, we rewrite \eqref{eq:change_var_V_n} as
\begin{align}
\overline{V}_N^D(t+\frac{1}{N}) & = 1+a\int_{0}^{t}\left(\int_0^{x}\frac{1}{\overline{V}_N^D(y)}dy  + \int_x^{\frac{\ceil{Nx}}{N}}\frac{1}{\overline{V}_N^D(y)}dy \right) dx +\nonumber \\
& \quad + a\int_{t}^{\frac{\floor{Nt}+1}{N}}\int_0^{\frac{\ceil{Nx}}{N}}\frac{1}{\overline{V}_N^D(y)}dy\ dx \nonumber\\
& = 1+a\int_{0}^{t}\int_0^{x}\frac{1}{\overline{V}_N^D(y)}dy\ dx + \int_0^t\int_x^{\frac{\ceil{Nx}}{N}}\frac{1}{\overline{V}_N^D(y)}dy\ dx + \nonumber\\
& \quad + a\int_{t}^{\frac{\floor{Nt}+1}{N}}\int_0^{\frac{\ceil{Nx}}{N}}\frac{1}{\overline{V}_N^D(y)}dy\ dx \nonumber\\
& = 1+a\int_{0}^{t}\int_0^{x}\frac{1}{\overline{V}_N^D(y)}dy\ dx + R_N(t)\label{eq:dynamics_Vn},
\end{align} where
\begin{align}
R_N(t):= \int_0^t\int_x^{\frac{\ceil{Nx}}{N}}\frac{1}{\overline{V}_N^D(y)}dy\ dx + a\int_{t}^{\frac{\floor{Nt}+1}{N}}\int_0^{\frac{\ceil{Nx}}{N}}\frac{1}{\overline{V}_N^D(y)}dy\ dx.\label{eq:R_N(t)}
\end{align}
By adding $\overline{V}_N^D(t)$ and subtracting $\overline{V}_N^D(t+\frac{1}{N})$ on both sides of \eqref{eq:dynamics_Vn}, we get
\begin{align}
\overline{V}_N^D(t) & = 1+a\int_0^t\int_0^x \frac{1}{\overline{V}_N^D(y)}dy\ dx + \overline{V}_N^D(t)-\overline{V}_N^D(t+\frac{1}{N})+ R_N(t).\label{eq:dynamics_Vn_1}
\end{align} Observe that the representation in \eqref{eq:dynamics_Vn_1} is equal to the one in \eqref{eq:mappingH}, as desired, if we define
\begin{align}
\overline{R}_N(t) := \overline{V}_N^D(t)-\overline{V}_N^D(t+\frac{1}{N})+R_N(t).\label{eq:overline_R_n}
\end{align}
For convenience, we introduce an operator $H$. This is useful for the proofs in Section \ref{subsubsec:contraction}.
\begin{definition}\label{def:mappingH}
Let $H:\mathbf{D}_{\geq 1}[0,1]\to\mathbf{D}_{\geq 1}[0,1]$ such that
\begin{align*}
(Hf)(t) = 1+a\int_0^t\int_0^s\frac{1}{f(u)}du\ ds,\ f\in \mathbf{D}_{\geq 1}[0,1].
\end{align*}
\end{definition}
This allows us to write the difference between $\overline{V}_N^D(t)$ and $V(t)$ as,
\begin{align*}
\overline{V}_N^D(t)-V(t) = \overline{R}_N(t) + (H\overline{V}_N^D)(t)-(HV)(t).
\end{align*}
So far, we reached Equation \eqref{eq:ideas_proof_conv} and made all the necessary preparation work to consider the remainder term $\overline{R}_N(t)$ and the term $(H\overline{V}_N^D)(t)-(HV)(t)$ of the right-hand side of \eqref{eq:ideas_proof_conv} in the next two sections, separately.

\subsubsection{Convergence of supremum of remainder term to zero}\label{subsubsec:convergence_of_remainder}
In this section, we show that the remainder term $\overline{R}_N(t)$ in \eqref{eq:overline_R_n} vanishes as $N\to\infty$, i.e., we show convergence of $\Sup{t\in[0,1]}|\overline{R}_N(t)|$ to zero as $N\to\infty$. Therefore, by \eqref{eq:overline_R_n}, we show the convergence of $\Sup{t\in[0,1]}|\overline{V}_N^D(t)-\overline{V}_N^D(t+\frac{1}{N})|\to 0$ as $N\to\infty$ in Lemma \ref{lemma:Vn-Vn+1/t} and $\Sup{t\in[0,1]}|{R}_N(t)|\to 0$ as $N\to\infty$ in Lemma \ref{lemma:Rn}. Then, as an immediate consequence of Lemmas \ref{lemma:Vn-Vn+1/t} and \ref{lemma:Rn}, we show $\Sup{t\in[0,1]}|\overline{R}_N(t)|\to 0$ as $N\to\infty$.

We begin with the convergence of $\Sup{t\in[0,1]}|\overline{V}_N^D(t)-\overline{V}_N^D(t+\frac{1}{N})|\to 0$ as $N\to\infty$ in Lemma \ref{lemma:Vn-Vn+1/t}.
\begin{lemma}\label{lemma:Vn-Vn+1/t}
Let $\overline{V}_N^D(t)$ as defined in \eqref{eq:mappingH}, then
\begin{align*}
\sup_{t\in[0,1]} \left|\overline{V}^D_N(t)-\overline{V}_N^D(t+\frac{1}{N})\right| \to 0\ \text{as}\ N\to\infty.
\end{align*}
\end{lemma}
\begin{proof} By definition of \eqref{eq:mappingH}, we have
\begin{multline}
\overline{V}_N^D(t)-\overline{V}_N^D(t+\frac{1}{N}) = a\bigg(\int_0^{\frac{\floor{Nt-1}+1}{N}}\int_0^{\frac{\ceil{Nx}}{N}}\frac{1}{\overline{V}_N^D(y)}dy\ dx-\\ -\int_0^{\frac{\floor{Nt}+1}{N}}\int_0^{\frac{\ceil{Nx}}{N}}\frac{1}{\overline{V}_N^D(y)}dy\ dx \bigg).\label{eq:first_remainder_term}
\end{multline} However, \eqref{eq:first_remainder_term} simplifies to,
\begin{align}
\overline{V}_N^D(t)-\overline{V}_N^D(t+\frac{1}{N}) = a\left(-\int_{\frac{\floor{Nt}}{N}}^{\frac{\floor{Nt}+1}{N}}\int_0^{\frac{\ceil{Nx}}{N}} \frac{1}{\overline{V}_N^D(y)}dy\ dx \right).\label{eq:first_remainder_term1}
\end{align} In what follows, we bound the absolute value of the right-hand side of \eqref{eq:first_remainder_term1} in terms of $N$. Since $\overline{V}_N^D(t)\in \mathbf{D}_{\geq 1}[0,1]$, we have, for all $t\in[0,1]$ that $\overline{V}_N^D(t)\geq 1$. Applying this bound, yields,
\begin{align*}
|\overline{V}_N^D(t)-\overline{V}_N^D(t+\frac{1}{N})| & \leq a\left(\int_{\frac{\floor{Nt}}{N}}^{\frac{\floor{Nt}+1}{N}}\int_0^{\frac{\ceil{Nx}}{N}} 1 dy\ dx \right)\\
& = a\left(\int_{\frac{\floor{Nt}}{N}}^{\frac{\floor{Nt}+1}{N}} \frac{\floor{Nt}+1}{N}dx \right).
\end{align*} Observe that, $\frac{\floor{Nt}+1}{N}-\frac{\floor{Nt}}{N}=\frac{1}{N}$. %for $n\in\mathbb{Z}_+$ and $z\in\mathbb{R}_+$, we have by definition $\ceil{nz}-\ceil{nz-1}=1$, which implies $\frac{\ceil{nz}}{n}-\frac{\ceil{nz-1}}{n}=\frac{1}{n}$.
Thus, for $t\in[0,1]$, we get
\begin{align*}
|\overline{V}_N^D(t)-\overline{V}_N^D(t+\frac{1}{N})| & \leq a\left(\frac{N+1}{N^2} \right).
\end{align*} As a consequence, we get the desired result,
\begin{align*}
\sup_{t\in[0,1]} |\overline{V}_N^D(t)-\overline{V}_N^D(t+\frac{1}{N})| \to 0\ \text{as}\ N\to\infty.
\end{align*}
\end{proof}
Now, it remains to be shown that $\Sup{t\in[0,1]}|{R}_N(t)|\to 0$ as $N\to\infty$. This is done in Lemma \ref{lemma:Rn}.
\begin{lemma}\label{lemma:Rn} Let $R_N(t)$ be defined as in \eqref{eq:R_N(t)}, then
\begin{align*}
\sup_{t\in[0,1]} |R_N(t)| \to 0\ \text{as}\ N\to\infty.
\end{align*}
\end{lemma}
\begin{proof}
Since $\overline{V}_N^D(t)\in \mathbf{D}_{\geq 1}[0,1]$, we have, for all $t\in[0,1]$ that $\overline{V}_N^D(t)\geq 1$, we get from \eqref{eq:R_N(t)}, that
\begin{align}
|R_N(t)| & \leq a\int_0^t\int_x^{\frac{\ceil{Nx}}{N}}1 dy\ dx + a\int_t^{\frac{\floor{Nt}+1}{N}}\int_0^{\frac{\ceil{Nx}}{N}}1 dy\ dx \nonumber\\
& = a\int_0^t\int_x^{\frac{\ceil{Nx}}{N}}1 dy\ dx + a\int_t^{\frac{\floor{Nt}+1}{N}} \frac{\ceil{Nx}}{N}\ dx.\label{eq:Rn_ineq}
\end{align} Observe that, for the first integral in \eqref{eq:Rn_ineq}, that $\frac{\ceil{Nx}}{N}-x \leq \frac{Nx+1}{N}-x = \frac{1}{N}$, and for the second integral, that $\frac{\floor{Nt}+1}{N}-t\leq \frac{Nt+1}{N}-t = \frac{1}{N}$. %for $n\in\mathbb{Z}_+$ and $z\in\mathbb{R}_+$, we have by definition $\ceil{nz}-nz\leq 1$, which implies $\frac{\ceil{nz}}{n}-z \leq \frac{1}{n}$.
Thus, for $t\in[0,1]$,
\begin{align*}
|R_N(t)| &\leq a\int_0^t \frac{1}{N}dx + a\int_t^{\frac{\floor{Nt}+1}{n}}\frac{\ceil{Nt}}{N}dx \nonumber\\
& = \frac{a}{N}+\frac{a}{N} = \frac{2a}{N}.
\end{align*} Hence,
\begin{align*}
\sup_{t\in[0,1]}|R_N(t)| \to 0\ \text{as}\ N\to\infty.
\end{align*}
\end{proof}
The convergence of the supremum of the remainder term $|\overline{R}_N(t)|$ follows immediately from Lemmas \ref{lemma:Vn-Vn+1/t} and \ref{lemma:Rn} and is summarized in the following corollary.
\begin{corollary}\label{cor:overline_Rn}
Let $\overline{R}_N(t)$ be defined as in \eqref{eq:overline_R_n}, then
\begin{align*}
\sup_{t\in[0,1]}|\overline{R}_N(t)|\to 0\ \text{as}\ N\to\infty.
\end{align*}
\end{corollary}

\subsubsection{Contractive property of map $H$}\label{subsubsec:contraction}
In this section, we show that $H:\mathbf{D}_{\geq 1}[0,h]\to \mathbf{D}_{\geq 1}[0,h]$ is a contractive mapping for a suitably chosen $h<1$. This allows us to show that $\Sup{0\leq t\leq h}|\overline{V}_N^D(t)-V(t)|\to 0$ as $N\to\infty$.

%\begin{lemma}
%Consider the initial value problem
%\begin{align}
%V''(t)=\frac{a}{V(t)} \label{eq:diffeq}
%\end{align} with
%\begin{align}
%V(0)=1,\ V'(0)=0.\label{eq:boundary_conditions}
%\end{align} 1he corresponding integral equation is given by
%\begin{align}
%V(t) = 1+a\int_0^t\int_0^{x} \frac{1}{V(y)}dy\ dx.\label{eq:integral_equation}
%\end{align}
%\end{lemma}
%\begin{proof}
%Integrating twice gives,
%\begin{align}
%V(t) & = 1+a\int_0^t\int_0^x \frac{1}{V(y)}dy\ dx\\
%& = 1+\int_0^t\int_0^x V''(y)dy\ dx\\
%& = 1+V(t)-V(0)-V'(0)t.\label{eq:integal_eq}
%\end{align} Now, inserting boundary conditions \eqref{eq:boundary_conditions} in \eqref{eq:integal_eq} yields the desired result.
%\end{proof}
%\begin{remark}\label{remark:solution_diff_eq}
%1he differential equation in \eqref{eq:diffeq} with initial conditions \eqref{eq:boundary_conditions} has the exact solution,
%\begin{align}
%V(t) = f_0(\sqrt{a}t)
%\end{align} where $f_0(\cdot)$ is given in Lemma \ref{}.
%\end{remark}

%From \eqref{eq:mappingH} we can write,
%\begin{align}
%\overline{V}_N^D(t) = (H\overline{V}_n^D)(t)+\overline{R}_n(t)
%\end{align} where

In Lemma \ref{lemma:contraction_H}, we show that $H$ is a contraction on $\mathbf{D}_{\geq 1}[0,h]$ for a suitably chosen $h<1$.
%\begin{align}
%\sup_{0\leq t\leq h}|\overline{V}_N^D(t)-V(t)|\to 0.
%\end{align}
\begin{lemma}\label{lemma:contraction_H}
Let $H$ as in Definition \ref{def:mappingH}. There exists $0\leq h<1$, such that $H$ is a contractive mapping on $\mathbf{D}_{\geq 1}[0,h]$, i.e.,
\begin{align*}
\sup_{0\leq t\leq h} |(H\overline{V}_n^D)(t)-(HV)(t)|\leq \kappa(h)\sup_{0\leq t\leq h} |\overline{V}_N^D(t)-V(t)|
\end{align*} with $\kappa(h)<1$.
\end{lemma}
\begin{proof} For $0\leq t\leq h$, we have
\begin{align*}
(H\overline{V}_N^D)(t)-(HV)(t) & = a \int_0^t\int_0^x \frac{1}{\overline{V}_N^D(y)}-\frac{1}{V(y)}dy\ dx \\
& = a \int_0^t\int_0^x \frac{V(y)-\overline{V}_N^D(y)}{\overline{V}_N^D(y)V(y)}dy\ dx
\end{align*} Since $\overline{V}_N^D(t)\in \mathbf{D}_{\geq 1}[0,1]$, we have, for all $t\in[0,1]$ that $\overline{V}_N^D(t)\geq 1$. Therefore we get,
\begin{align*}
|(H\overline{V}_n^D)(t)-(HV)(t)| & \leq a \int_0^t\int_0^x |V(y)-\overline{V}_N^D(y)|dy\ dx.
\end{align*} Furthermore, by definition of the supremum, we get
\begin{align*}
|(H\overline{V}_n^D)(t)-(HV)(t)| & \leq a \int_0^t\int_0^x \sup_{0\leq y\leq h} |V(y)-\overline{V}_N^D(y)| dy\ dx\\
& = a\sup_{0\leq y\leq h} |V(y)-\overline{V}_N^D(y)|\int_0^t\int_0^x 1 dy\ dx.
\end{align*} Thus, for $0\leq t\leq h$,
\begin{align*}
|(H\overline{V}_n^D)(t)-(HV)(t)| & \leq \frac{ah^2}{2}\sup_{0\leq y\leq h} |V(y)-\overline{V}_N^D(y)|.
\end{align*} Hence,
\begin{align*}
\sup_{0\leq t\leq h} |(H\overline{V}_n^D)(t)-(HV)(t)|\leq \kappa(h)\sup_{0\leq t\leq h} |V(t)-\overline{V}_N^D(t)|
\end{align*} where $\kappa(h)=\frac{ah^2}{2}$. If $a>2$, take $h<\sqrt{\frac{2}{a}}$. Then $\kappa(h) = \frac{ah^2}{2}<1$. Otherwise, if $a\leq 2$, then $\kappa(h) = \frac{ah^2}{2} \leq h^2 < 1$, since $h<1$.
\end{proof}
From the proof of Lemma \ref{lemma:contraction_H}, it follows that the contraction property of the map $H$ on the interval $[0,h]$ depends on the size of $a$. We choose
\begin{align}
h<\begin{cases}
\sqrt{\frac{2}{a}} & \text{if}\  a>2, \\
1 & \text{if}\  a\leq 2.
\end{cases}\label{eq:suitably_h}
\end{align} Recall from \eqref{eq:k_N} that $a = \hat{r}\lambda_N^DN^2\geq 0$. Thus, for every $a\geq 0$, we choose a suitable $h$, according to \eqref{eq:suitably_h}, such that $H$ is a contraction on the interval $[0,h]$.

\subsubsection{Convergence of supremum on different intervals of $\overline{V}_N^D(t)$ to $V(t)$}\label{subsubsec:conv_on_diff_intervals}
In Sections \ref{subsubsec:convergence_of_remainder} and \ref{subsubsec:contraction}, we showed that the remainder term $\overline{R}_N(t)$ vanishes as $N\to\infty$ for every $t$ in the interval $[0,1]$ and that the map $H$ is a contractive mapping on the interval $[0,h]$, respectively. Using both results yields the convergence of $\Sup{0\leq t\leq h}|\overline{V}_N^D(t)-V(t)| \to 0$ as $N\to\infty$ in Lemma \ref{lemma:case_m=1}. Then, we show, by induction, that $\Sup{mh\leq t\leq (m+1)h}|\overline{V}_N^D(t)-V(t)|\to 0$ as $N\to\infty$ for every $m=0,\ldots,M$ such that $Mh=1$ in Lemma \ref{lemma:case_general_m}. This allows us to show that $\Sup{0\leq t\leq 1}|\overline{V}_N^D(t)-V(t)|\to 0$ as $N\to\infty$, as desired in Proposition \ref{THM:CONVERGENCE}.

Using the contraction property as shown in Lemma \ref{lemma:contraction_H}, we show the convergence of $\overline{V}_N^D(t)$ to $V(t)$ in the supremum norm.
\begin{lemma}\label{lemma:case_m=1}
Let $h$ as in \eqref{eq:suitably_h} and $\overline{V}_N^D(t)$ as defined in \eqref{eq:mappingH} and $V(t)$ as in \eqref{eq:integral_equation_V}. Then we have,
\begin{align*}
\sup_{0\leq t\leq h} |\overline{V}_N^D(t)-V(t)|\to 0,\  \text{as}\  N\to\infty.
\end{align*}
\end{lemma}
\begin{proof}
By Equation \eqref{eq:ideas_proof_conv} and the triangle inequality, we have
\begin{align}
\sup_{0\leq t\leq h} |\overline{V}_N^D(t)-V(t)| & = \sup_{0\leq t\leq h} |\overline{R}_N(t)+(H\overline{V}_N^D)(t)-(HV)(t)| \nonumber\\
& \leq \sup_{0\leq t\leq h}|\overline{R}_N(t)|+\sup_{0\leq t\leq h}|(H\overline{V}_N^D)(t)-(HV)(t)|. \label{eq:contraction_application}
\end{align} %Using the inequalities in \eqref{eq:ineq_Vn-Vn-1} and \eqref{eq:ineq_Rn} in Lemmas \ref{lemma:Vn-Vn+1/t} and \ref{lemma:Rn} we find,
%\begin{align}
%\sup_{0\leq t\leq h} |\overline{V}_N^D(t)-V(t)| & \leq \sup_{0\leq t\leq h} \left(a\left(\frac{N+1}{N^2}\right)+\frac{2a}{N}\right)+\sup_{0\leq t\leq h}|(H\overline{V}_N^D)(t)-(HV)(t)|\\
%& = \left(a\left(\frac{N+1}{N^2}\right)+\frac{2a}{N}\right)+\sup_{0\leq t\leq h}|(H\overline{V}_N^D)(t)-(HV)(t)|.
%\end{align}
Using Lemma \ref{lemma:contraction_H} in the right-hand side of \eqref{eq:contraction_application}, we get
\begin{align*}
\sup_{0\leq t\leq h} |\overline{V}_N^D(t)-V(t)| & \leq \sup_{0\leq t\leq h}|\overline{R}_N(t)|+\kappa(h)\sup_{0\leq t\leq h}|\overline{V}_N^D(t)-V(t)|.
\end{align*} Rearranging terms yield,
\begin{align*}
\sup_{0\leq t\leq h}|\overline{V}_N^D(t)-V(t)|\leq \frac{1}{1-\kappa(h)}\sup_{0\leq t\leq h}|\overline{R}_N(t)|.
\end{align*} Remark that $\sup_{0\leq t\leq h}|\overline{R}_N(t)| \leq \sup_{0\leq t\leq 1}|\overline{R}_N(t)|$. Hence, by Corollary \ref{cor:overline_Rn},
\begin{align*}
\sup_{0\leq t\leq h}|\overline{V}_N^D(t)-V(t)|\to 0\ \text{as}\ N\to\infty.
\end{align*}
\end{proof}
Now, we are in position to prove that on every small interval $[mh,(m+1)h]$ (where $m=0,\ldots,M$ such that $Mh=1$) the supremum over $t\in[mh,(m+1)h]$ of the absolute difference between $\overline{V}_N^D(t)$ and $V(t)$ goes to zero as $N\to\infty$.
\begin{lemma}\label{lemma:case_general_m}
Let $h$ as in \eqref{eq:suitably_h}. For every $0\leq m\leq M$, let $t\in[mh,(m+1)h]$. Then
\begin{align}
\sup_{mh\leq t\leq (m+1)h} |\overline{V}_N^D(t)-V(t)|\to 0,\ \text{as}\  N\to\infty.\label{eq:induction_sup}
\end{align}
\end{lemma}
\begin{proof} We prove the statement in \eqref{eq:induction_sup} by strong induction. By Lemma \ref{lemma:case_m=1}, the statement holds for $m=0$; i.e.,
\begin{align*}
\sup_{0\leq t\leq h}|\overline{V}_N^D(t)-V(t)|\to 0\ \text{as}\  N\to\infty.
\end{align*} Suppose that the statement holds for $0\leq t\leq mh$:
\begin{align}
\sup_{0\leq t\leq mh} |\overline{V}_N^D(t)-V(t)| \to 0\ \text{as}\ N\to\infty.\label{eq:induction_step_m+1}
\end{align} Then, we need to show that for $mh\leq t\leq (m+1)h$, the statement in \eqref{eq:induction_sup} holds. Before we move on with expressions of the form $\sup_{mh\leq t\leq (m+1)h}|\overline{V}_N^D(t)-V(t)|$, we first try to bound the expression $|\overline{V}_N^D(t)-V(t)|$ itself. By Equations \eqref{eq:integral_equation_V} and \eqref{eq:continuous_extension_VD}, we have
\begin{align*}
\overline{V}_N^D(t)-V(t) & = a\int_0^t\int_0^x \frac{1}{\overline{V}_N^D(y)}-\frac{1}{V(y)}dy\ dx + \overline{R}_N(t) \nonumber \\
& = a\int_0^{mh}\int_0^x \frac{1}{\overline{V}_N^D(y)}-\frac{1}{V(y)}dy\ dx +\nonumber\\
& \quad + a\int_{mh}^{t}\int_0^x \frac{1}{\overline{V}_N^D(y)}-\frac{1}{V(y)}dy\ dx + \overline{R}_N(t),
\end{align*} and since $\overline{V}_N^D(t)\in \mathbf{D}_{\geq 1}[0,1]$, we have, for all $t\in[0,1]$ that $\overline{V}_N^D(t)\geq 1$. Therefore we get,
\begin{align}
|\overline{V}_N^D(t)-V(t)| = & \leq a\int_0^{mh}\int_0^x |V(y)-\overline{V}_N^D(y)| dy\ dx + \nonumber\\
& \quad + a\int_{mh}^{t}\int_0^x |V(y)-\overline{V}_N^D(y)|dy\ dx + |\overline{R}_N(t)|.\label{eq:diff_vnt1}
\end{align}
In what follows, we find for each term in \eqref{eq:diff_vnt1} an upper bound in either terms of $\sup_{0\leq t\leq mh}|\overline{V}_N^D-V(t)|$ or $\sup_{mh\leq t\leq (m+1)h}|\overline{V}_N^D(t)-V(t)|$ or terms that vanish as $N\to\infty$. Then, by rearranging terms and using the induction step in \eqref{eq:induction_step_m+1}, we find that the statement in \eqref{eq:induction_sup} holds.

Consider the first term on the right-hand side of \eqref{eq:diff_vnt1}. By definition of the supremum, we get
\begin{align*}
a\int_0^{mh}\int_{0}^x |V(y)-\overline{V}_N^D(y)|dy\ dx & \leq a\int_0^{mh}\int_0^x \sup_{0\leq y\leq mh}|V(y)-\overline{V}_N^D(y)| dy\ dx \nonumber \\
& = \frac{a(mh)^2}{2}\sup_{0\leq y\leq mh}|V(y)-\overline{V}_N^D(y)|.
\end{align*} Now, consider the second term on the right-hand side of \eqref{eq:diff_vnt1}. %By Lemmas \ref{lemma:Vn-Vn+1/t} and \ref{lemma:Rn} we have
%\begin{align}
%|\overline{R}_N(t)|\leq a\frac{\floor{N(m+1)h}}{N^2}+\frac{a(m+1)h}{N}+\frac{a\ceil{N(m+1)h}}{N^2}.\label{eq:RN_ineq}
%\end{align}
This integral can be bounded as follows;
\begin{align*}
a\int_{mh}^{t}\int_0^x |V(y)-\overline{V}_N^D(y)|dy\ dx & = a\int_{mh}^t\int_0^{mh} |V(y)-\overline{V}_N^D(y)|dy\ dx \nonumber \\ & \quad + a\int_{mh}^t\int_{mh}^x |V(y)-\overline{V}_N^D(y)|dy\ dx \\
& \leq ah(mh)\sup_{0\leq y\leq mh}|V(y)-\overline{V}_N^D(y)| \nonumber\\ & \quad +\frac{ah^2}{2}\sup_{mh\leq y\leq (m+1)h}|V(y)-\overline{V}_N^D(y)|
\end{align*} Finally, the third term on the right-hand side of \eqref{eq:diff_vnt1} is the remainder term $|\overline{R}_N(t)|$. Continuing from \eqref{eq:diff_vnt1} and taking the supremum over the interval $m\leq t\leq(m+1)h$ of $|\overline{V}_N^D(t)-V(t)|$ on the left-hand side, yields,
\begin{multline*}
\sup_{mh\leq t\leq (m+1)h} |\overline{V}_N^D(t)-V(t)| \leq \frac{a(mh)^2}{2}\sup_{0\leq y\leq mh}|V(y)-\overline{V}_N^D(y)|+\\ +amh^2\sup_{0\leq y\leq mh}|V(y)-\overline{V}_N^D(y)| + \frac{ah^2}{2}\sup_{mh\leq y\leq (m+1)h}|V(y)-\overline{V}_N^D(y)| + |\overline{R}_N(t)|
\end{multline*} Simplifying and rearranging terms yields,
\begin{multline*}
\sup_{mh\leq t\leq (m+1)h} |\overline{V}_N^D(t)-V(t)| \leq \frac{1}{1-\frac{1}{2}ah^2} \bigg(amh^2\left(\frac{m}{2}+1\right)\sup_{0\leq y\leq mh}|V(y)-\overline{V}_N^D(y)|+\\ +|\overline{R}_N(t)| \bigg)
\end{multline*} Remark that $\sup_{mh\leq t\leq (m+1)h}|\overline{R}_N(t)| \leq \sup_{0\leq t\leq 1}|\overline{R}_N(t)|$. Hence, by Corollary \ref{cor:overline_Rn} and from the induction hypothesis that $
\Sup{0\leq t\leq mh}|\overline{V}_N^D(t)-V(t)|\to 0\ \text{as}\ N\to\infty$, it follows that
\begin{align*}
\sup_{mh\leq t\leq (m+1)h} |\overline{V}_N^D(t)-V(t)|\to 0\ \text{as}\ N\to\infty.
\end{align*}
\end{proof}
%Now we are ready to prove the statement we are ultimately looking for; i.e., $\Sup{0\leq t\leq 1}|\overline{V}_N^D(t)-V(t)|\to 0$ as $N\to\infty$.
Now, we have all the ingredients to prove Theorem \ref{THM:CONVERGENCE}.
%\begin{lemma}\label{lemma:H_contractive_mapping} For $0\leq t\leq 1$ and $0\leq h < \sqrt{\frac{2}{a}}$,
%\begin{align}
%\sup_{0\leq t\leq 1} |\overline{V}_N^D(t)-V(t)|\to 0,\ \text{as}\  N\to\infty.
%\end{align}
%\end{lemma}
%\begin{proof}
%Consider the partition $\{[0,h],[h,2h],\ldots,[(M-1)h,Mh]\}$ of the interval $[0,1]$. Observe that,
%\begin{align}
%& \sup_{0\leq t\leq 1} |\overline{V}_N^D(t)-V(t)|\\
%& \leq \sup_{0\leq t\leq h} |\overline{V}_N^D(t)-V(t)| + \sup_{h\leq t\leq 2h} |\overline{V}_N^D(t)-V(t)| + \ldots + \sup_{(M-1)h\leq t\leq Mh} |\overline{V}_N^D(t)-V(t)|.
%\end{align}
%\end{proof}

\begin{proof}[Proof of Proposition \ref{THM:CONVERGENCE}]
Let $t\in[0,1]$ and $\overline{V}_N^D(t)$ as defined in \eqref{eq:mappingH} and $V(t)$ as in \eqref{eq:integral_equation_V}. Furthermore, let $h$ as in \eqref{eq:suitably_h} and consider the partition $\{[0,h],[h,2h],\ldots,[(M-1)h,Mh]\}$ of the interval $[0,1]$. Then,
\begin{multline*}
\sup_{0\leq t\leq 1}|\overline{V}_N^D(t)-V(t)| \leq \sup_{0\leq t\leq h} |\overline{V}_N^D(t)-V(t)| +\\ + \sup_{h\leq t\leq 2h} |\overline{V}_N^D(t)-V(t)| + \cdots + \sup_{(M-1)h\leq t\leq Mh} |\overline{V}_N^D(t)-V(t)|.
\end{multline*}

%From Corollary \ref{cor:overline_Rn} we have $\Sup{t\in[0,1]}|\overline{R}_N(t)|\to 0$ as $N\to\infty$, and from Lemma \ref{lemma:contraction_H} we have, for $h<\sqrt{\frac{2}{a}}$ that the mapping $H$ is a contraction. Then, using Lemmas \ref{lemma:case_m=1} and \ref{lemma:case_general_m}, we finally show in Lemma \ref{lemma:H_contractive_mapping} that
%\begin{align}
%\sup_{0\leq t\leq 1}|\overline{V}_N^D(t)-V(t)|\to 0\ \text{as}\ N\to\infty.
%\end{align}
By Lemma \ref{lemma:case_general_m}, we have for the supremum over each interval $[mh,(m+1)h]$, for $m=0,\ldots,M-1$, of the absolute difference of $\overline{V}_N^D(t)$ and $V(t)$ that it converges to zero as $N\to\infty$. Hence,
\begin{align*}
\sup_{0\leq t\leq 1}|\overline{V}_N^D(t)-V(t)|\to 0\ \text{as}\ N\to\infty.
\end{align*}
\end{proof}

\subsection{Numerical validation of convergence in Proposition \ref{THM:CONVERGENCE}}

We validate the convergence established in Theorem \ref{THM:CONVERGENCE}. We run the recursion $V_n^D, n=0,\ldots,N$ until node $N$ and evaluate the approximation $V(t)$ in the point $t=1$ for different values of the parameter $a$. The results obtained via the approximation are compared with the recursion values, i.e., we compare $V(1)$ (computed via \eqref{eq:V1_approx} where the function $f_0$ is defined in Lemma \ref{lemma:solution_f}) with $V_N^D$ for different values of $a$. For each value of $a$, we compute $V(1)$ and $V_N^D$ for several values of $N$, ranging from $10$ to $10^4$. One example of $V(1)$ and $V_N^D$, where $a=0.05$ is shown in Figure \ref{fig:convergence_V1} and Table \ref{tab:convergence_VnD}. In the case that $N$ is small, e.g. $N=10$, the absolute and relative error are only $\approx 0.2\%$.
\begin{figure}
\centering
\includegraphics[scale=0.5]{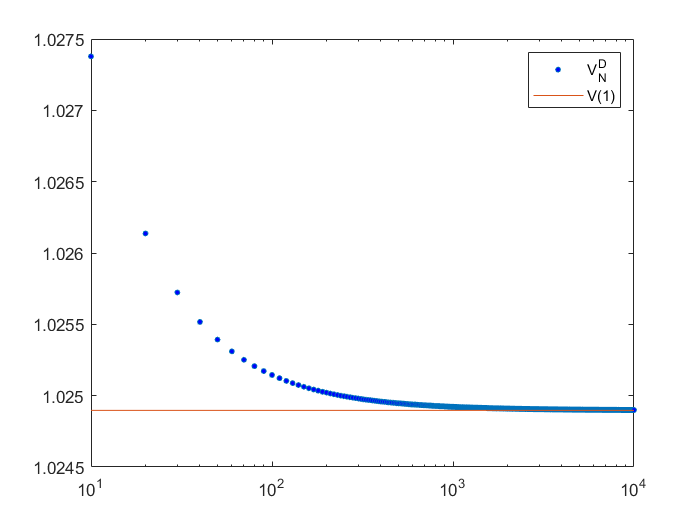}
\caption{Recursion values $V_N^D$ (blue solid circles) and approximation $V(1)$ (red solid line), for $a=0.05$.}
\label{fig:convergence_V1}
\end{figure}

\begin{table}[h!]
\begin{center}
 \begin{tabular}{|c r r c c|}
 \hline
 \multicolumn{5}{|c|}{$a=0.05$} \\
 \hline
 $N$ & $V_N^D$ & $V(1)$ & $|V(1)-V_N^D|$ & $\frac{|V(1)-V_N^D|}{V_N^D}$  \\ [0.5ex]
 \hline\hline
 10 & 1.02737778 & 1.02489702 & 0.00248075 & 0.00241464\\
 \hline
 $10^2$ & 1.02514499 & 1.02489702 & 0.00024188 & 0.00297667\\
 \hline
 $10^3$ & 1.02492182 & 1.02489702 & 0.00002419 & 0.00029824\\
 \hline
 $10^4$ & 1.02489950 & 1.02489702 & 0.00000241& 0.00002983\\
 \hline
 $10^5$ & 1.02489728 & 1.02489702 & 0.00000024 & 0.00000298\\
 \hline
\end{tabular}
\end{center}
\caption{Absolute and relative error of different values of $N$ for $a=0.05$.}
\label{tab:convergence_VnD}
\end{table}

One can see from Figure \ref{fig:convergence_V1} that, as $N$ gets bigger, $V_N^D$ decreases to $V(1)$. Thus, the approximation $V(1)$ offers a close approximation of the voltage $V_N^D$ under the Distflow model, especially for large $N$. The errors are bigger for small $N$.

\section{Integral equation}\label{sec:integral_equation}
\begin{lemma}\label{lemma:solution_f}
For $t\geq 0,k>0,y>0,w\geq 0$, the nonlinear differential equation
\begin{align*}
f''(t) = \frac{k}{f(t)}
\end{align*} with initial conditions $f(0)=y$ and $f'(0)=w$ has the exact solution $f(t) = \gamma f_0(\alpha+\beta t)$. Here, $f_0$ is given by
\begin{align*}
f_0(x) = \exp(U^2(x)),\quad \text{for}~x\geq 0,
\end{align*} where $U(x)$, for $x\geq 0$, is given by
\begin{align}\label{eq:Ux}
\int_0^{U(x)}\exp(u^2)~du = \frac{x}{\sqrt{2}},
\end{align}and
\begin{align*}
\alpha & = \sqrt{2}\int_0^\frac{w}{\sqrt{2k}} \exp(u^2)~du,\\
\beta & = \frac{\sqrt{k}}{y}\exp\left(\frac{w^2}{2k}\right),\\
\gamma & = y\exp\left(\frac{-w^2}{2k} \right).
\end{align*}
\label{LEMMA:DIFF_EQ1}
\end{lemma}

%First, notice that in our continuous analogue described by Equations \eqref{eq:de_f} and \eqref{eq:de_f_initial}, we have, next to the initial condition $f(0)=1$, the initial condition $f(1)=1+k$, while in Lemma \ref{LEMMA:DIFF_EQ1} we have the initial conditions $f(0)=y$ and $f'(0)=w$. Of course, for the asymptotic behavior of the solution $f$ we are not interested in what happens at the value $t=1$, but we should connect the conditions $f'(0)=w$ and $f(1)=1+k$. Actually, we can show the existence and uniqueness of $w\geq 0$ such that the solution of $f''(t)f(t)=k,t\geq 0; f(0)=1,f'(0)=w$ satisfies $f(1)=1+k$. See Appendix \ref{appendix:unique_w} for the proof.

\begin{proof}[Proof of Lemma \ref{LEMMA:DIFF_EQ1}]
From the differential equation
\begin{align}
f''(t) = 1/f(t),\ t\geq 0 \label{eq:diff_f_app}
\end{align} we get
\begin{align}
f'(u)f''(u) = f'(u)/f(u),\ 0\leq u\leq t.\label{eq:f4_app}
\end{align} Integrating Equation \eqref{eq:f4_app} over $u$ from 0 to $t$ using $f(0)=1, f'(0)=0$ we get
\begin{align*}
\int_0^t f'(u)f''(u)du = \frac{1}{2}(f'(t))^2 = \int_0^t \frac{kf'(u)}{f(u)} = \ln(f(t)).
\end{align*} Hence, for $t>0$,
\begin{align}
\frac{f'(t)}{\left(2\ln(f(t))\right)^{\frac{1}{2}}} = 1.\label{eq:alternative_f_prime}
\end{align}
Integrating $f'(u)/(2\ln(f(t)))^{\frac{1}{2}})=1$ from $u=0$ to $u=t$, while substituting $s=f(u)\in [1,f(t)]$, we get
\begin{align}
%\int_0^t \frac{\frac{df}{ds}\frac{ds}{du}}{(w^2+2k\ln(f(u)))^{\frac{1}{2}}} du & \nonumber =
\int_1^{f(t)} \frac{1}{(2\ln(s))^{\frac{1}{2}}}ds = t.\label{eq:alternative_rep_f}
\end{align} Substituting $v=(\ln(s))^{\frac{1}{2}}, s=\exp(v^2),ds=2v\exp(v^2)dv$ in the integral \eqref{eq:alternative_rep_f}, we get
\begin{align*}
\int_0^{(\ln(f(t))^{\frac{1}{2}}}\frac{2v\exp(v^2)}{\sqrt{2}v} dv = t,
\end{align*} i.e.,
\begin{align*}
\int_0^{(\ln(f(t))^{\frac{1}{2}}}\exp(v^2)dv = t/\sqrt{2}.
\end{align*} Hence, when we define $U(t)$ by
\begin{align}
\int_0^{U(t)}\exp(v^2)dv = t/\sqrt{2},\label{eq:Ut}
\end{align} we have
\begin{align}
(\ln(f(t))^{\frac{1}{2}}=U(t), \text{i.e.}, f(t)=\exp(U^2(t)).\label{eq:ln(ft)}
\end{align} Observe that from \eqref{eq:alternative_f_prime} and \eqref{eq:ln(ft)} we get,
\begin{align}
f'(t) = \sqrt{2}U(t), t\geq 0.\label{eq:f'(t)}
\end{align} We denote in the sequel the solution of $f''(t)=1/f(t), t\geq 0$, with $f(0)=1,f'(0)=0$ by $f_0(t)$.
Next, for given $w\geq 0$ and $y,k>0$ we consider the following initial value problem:
\begin{align}
f''(t) = \frac{k}{f(t)}~\text{for}~ t\geq 0; f(0)=y\ \text{and}\ f'(0)=w.\label{eq:ivp_equation2}
\end{align}
We find unique, explicit and positive values $\gamma,\alpha$ and $\beta$ such that
\begin{align}
f(t) = \gamma f_0(\alpha+\beta t)~\text{for}~t\geq 0. \label{eq:sol_ft}
\end{align} The function $f(t)$ in \eqref{eq:sol_ft} fulfills the conditions in \eqref{eq:ivp_equation2},
\begin{align*}
\begin{cases}
f''(t) = {\beta}^2{\gamma}^2f_0^{''}(\alpha+\beta t)f_0(\alpha+\beta t) \overset{\eqref{eq:diff_f_app}}{=} {\beta}^2{\gamma}^2 = k,\\
f(0) = \gamma f_0(\alpha) = y,\\
f'(0) = \beta\gamma f_0^{'}(\alpha) = w,
\end{cases}
\end{align*} i.e., if and only if
\begin{align*}
\begin{cases}
\gamma\beta = \sqrt{k},\\
\gamma f_0(\alpha) = y,\\
f_0^{'}(\alpha) = \frac{w}{\sqrt{k}}.
\end{cases}
\end{align*}
%Furthermore, differentiation with respect to the variable $x$ of \eqref{eq:ivp_solution} and \eqref{eq:ivp_Ux}, gives
%\begin{align*}
%f_0^{'}(x) & = 2U(x)\exp(U(x)^2)U'(x) \\
%\frac{d}{dx}\left(\int_0^{U(x)}\exp(u^2)du \right)& = \exp(U(x)^2)U'(x) = %\frac{1}{\sqrt{2}}
%\end{align*} Combining the above results, results in an expression for $f_0^{'}(x)$, that is,
%\begin{align*}
%f_0^{'}(x) = \sqrt{2}~U(x)
%\end{align*}
%1hus, to fulfill the conditions in \eqref{eq:conditions}, we need
%\begin{align}
%f_0^{'}(a) = \sqrt{2}U(a) = \frac{w}{\sqrt{z}},\  \text{i.e.,}\  U(a)=\frac{w}{\sqrt{2z}}
%\end{align} and by \eqref{eq:Ut} we find
From $f_0'(\alpha)=\frac{w}{\sqrt{2k}}$ we get $U(\alpha)=\frac{w}{\sqrt{2k}}$ by \eqref{eq:f'(t)}, and $w$, from \eqref{eq:Ut}, we find
\begin{align*}
\int_0^\frac{w}{\sqrt{2k}} \exp(u^2)du = \frac{\alpha}{\sqrt{2}}.
\end{align*} Hence, subsequently we get
\begin{align*}
\alpha & = \sqrt{2}\int_0^\frac{w}{\sqrt{2k}} \exp(u^2)du, \\
\gamma & = \frac{y}{f_0(\alpha)} = \frac{y}{\exp(U(\alpha)^2)} = \frac{y}{\exp\left(\frac{w^2}{2k}\right)}, \\
\beta & = \frac{\sqrt{k}}{\frac{y}{\exp\left(\frac{w^2}{2k}\right)}} = \frac{\sqrt{k}}{y}\exp\left(\frac{w^2}{2k}\right).
\end{align*}
\end{proof}

Notice that we do not find an elementary closed-form solution of the function $f$, since $f$ is given in terms of $U(x)$. The function $U(x)$, for $x\geq 0$, is given by Equation \eqref{eq:Ux}. The left-hand side of \eqref{eq:Ux} is equal to $\frac{1}{2}\sqrt{\pi} \text{erfi}(U(x))$ where  $\text{erfi}(z)$ is the imaginary error function.

\section{Comparison of stability regions}\label{sec:appendix_comparison}
The critical arrival rates for the Distflow and the Linearized Distflow models are given in \eqref{eq:critical_lambda} and  \eqref{eq:critical_lambda_L}, respectively. To compare these critical rates, we use \eqref{eq:P(delta)} and show that $P(\Delta)$ is a strictly decreasing function.
\begin{proof}
We let
\begin{align*}
P(\Delta) = 2Q^2(\Delta),
\end{align*} where
\begin{align*}
Q(\Delta) := \frac{1-\Delta}{\sqrt{1-(1-\Delta)^2}}\int_0^{\sqrt{\ln\left(\frac{1}{1-\Delta} \right)}}\exp(u^2)du,\ 0< \Delta\leq\frac{1}{2}.
\end{align*} Set $y=\sqrt{\ln\left(\frac{1}{1-\Delta} \right)}\in[0,\infty)$. Then,
\begin{align*}
Q(\Delta) = \frac{\exp(-y^2)}{\sqrt{1-\exp(-2y^2)}}\int_0^y \exp(u^2)du  = \frac{1}{\sqrt{\exp(2y^2)-1}}\int_0^y \exp(u^2)du =: H(y).
\end{align*} Since $y$ is a strictly increasing function of $\Delta\in[0,\frac{1}{2}]$, it is sufficient to show that $H(y)$ is a strictly decreasing function of $y$. We compute
\begin{align*}
H'(y) %& = -\frac{1}{2}(\exp(2y^2)-1)^{-\frac{3}{2}}4y\exp(2y^2)\int_0^y \exp(u^2)du + \frac{\exp(y^2)}{\sqrt{\exp(2y^2)-1}}\\
& = (\exp(2y^2)-1)^{-\frac{1}{2}}\left(-\frac{2y\exp(2y^2)}{\exp(2y^2)-1}\int_0^y \exp(u^2) du+\exp(y^2) \right).
\end{align*} Therefore,
\begin{align*}
H'(y)<0 \iff \int_0^y \exp(u^2) > \frac{\exp(y^2)}{\frac{2y\exp(2y^2)}{\exp(2y^2)-1}} = \frac{\sinh(y^2)}{y}.
\end{align*} %i.e., since
%\begin{align}
%\frac{\exp(y^2)}{\frac{2y\exp(2y^2)}{\exp(2y^2)-1}} = \frac{\exp(y^2)-\exp(-y^2)}{2y} = \frac{\sinh(y^2)}{y}.
%\end{align}
%Thus,
%\begin{align}
%H'(y)<0 \iff \int_0^y \exp(u^2) du > \frac{\sinh(y^2)}{y}.
%\end{align}
We have,
\begin{align*}
\int_0^y \exp(u^2)du\Big|_{y=0} = 0 = \frac{\sinh(y^2)}{y}\Big|_{y=0},
\end{align*} so it is enough to show that
\begin{align*}
\frac{d}{dy}\left(\int_0^y \exp(u^2)du \right) = \exp(y^2) > \frac{d}{dy}\left(\frac{\sinh(y^2)}{y} \right),\ y>0.
\end{align*} We compute
\begin{align*}
\frac{d}{dy}\left(\frac{\sinh(y^2)}{y} \right) & = \frac{2y\cosh(y^2)y-\sinh(y^2)}{y^2} \\
& = 2\cosh(y^2)-\frac{\sinh(y^2)}{y^2} = \exp(y^2)+\exp(-y^2)-\frac{\sinh(y^2)}{y^2},
\end{align*} and so, it is sufficient to show
\begin{align*}
\exp(-y^2)-\frac{\sinh(y^2)}{y^2}<0, y>0.
\end{align*} We have,
\begin{align*}
\exp(-y^2)-\frac{\sinh(y^2)}{y^2} %& = \exp(-y^2)-\frac{\exp(y^2)-\exp(-y^2)}{2y^2} \\
%& = \exp(-y^2)\left(1-\frac{\exp(2y^2)-1}{2y^2} \right) \\
%& = -\exp(-y^2)\left(\frac{\exp(2y^2-1}{2y^2}-1 \right) \\
& = -\exp(-y^2)\frac{\exp(2y^2)-1-2y^2}{2y^2} < 0,
\end{align*} since $\exp(t)-1-t>0$ for $t>0$. Furthermore,
$P(0) = 1$ since $H(y)\to\frac{1}{\sqrt{2}}$ as $y\to 0$ and $P(\frac{1}{2})=\frac{\pi}{6}\text{erfi}\left(\sqrt{\ln(2)} \right)^2\approx 0.77$.
\end{proof}

\section{Notation}\label{sec:notation}
\begin{itemize}\setlength\itemsep{-1em}
\item $N$: number of charging stations
\item $\lambda$: arrival rate of EVs at each charging station
\item $\boldsymbol{\lambda} = (\lambda,\ldots,\lambda)$: the arrival rate at each charging station
\item $\lambda_c$: critical arrival rate for EVs
\item $\lambda_c^D$: critical arrival rate under the Distflow model
\item $\lambda_c^L$: critical arrival rate under the Linearized Distflow model
\item $\mathbf{X}(t)=(X_1(t),\ldots,X_N(t))$: the number of EVs at each charging station at time $t$
\item $\mathbf{p}=(p_1(t),\ldots,p_N(t))$: the power allocated to each charging station at time $t$
\item $g(\cdot),h(\cdot)$: auxiliary functions to define utility-maximizing mechanism
\item $\mathcal{C}$: $N$-dimensional vector space that contains the distribution network constraints
\item $\mathcal{G}=(\mathcal{I},\mathcal{E})$: directed graph
\item $\mathcal{I}$: set of nodes
\item $\mathcal{E}$: set of edges
\item $\epsilon_{ij}$: edge
\item $z$: impedance on edge
\item $r$: resistance on edge
\item $x$: reactance on edge.
%\item $\hat{r}$: defined as $\hat{r}=\frac{r+x}{r^2+x^2}$.
\item $\tilde{V}_j$: real voltage at node $j\in\mathcal{I}$
\item $\tilde{V}_j^L$: real voltage at node $j\in\mathcal{I}$ under Linearized Distflow model
\item $\tilde{V}_j^D$: real voltage at node $j\in\mathcal{I}$ under Distflow model
\item $W_{ij}$: transformed voltage; product of real voltage at node $i,j\in\mathcal{I}$ (after relabeling)
\item $W_{ij}^L$: transformed voltage; product of real voltage at node $i,j\in\mathcal{I}$ under Linearized Distflow model (after relabeling)
\item $W_{ij}^D$: transformed voltage; product of real voltage at node $i,j\in\mathcal{I}$ under Distflow model (after relabeling)
\item $\tilde{s}_i$: complex power consumption at node $i\in\mathcal{I}$
\item $\tilde{p}_i$: active power consumption at node $i\in\mathcal{I}$
\item $\tilde{q}_i$: reactive power consumption at node $i\in\mathcal{I}$
\item $I_{ij}$: complex current on edge $\epsilon_{ij}\in\mathcal{E}$
\item $\tilde{S}_{ij}$: complex power flowing over edge $\epsilon_{ij}\in\mathcal{E}$
\item $\tilde{P}_{ij}$: active power flowing over edge $\epsilon_{ij}\in\mathcal{E}$
\item $\tilde{Q}_{ij}$: reactive power flowing over edge $\epsilon_{ij}\in\mathcal{E}$
\item $\Delta$: bound on the maximal voltage drop
\item $k_N$: product of $r$ and arrival rate $\lambda_c^D$.
\end{itemize}
\end{appendices}

\bibliographystyle{apa}
\bibliography{Comparison_of_stability_regions_for_a_line_distribution_network_with_stochastic_load_demands}
\end{document}